\documentclass[11pt,USEnglish,a4paper]{article}
\usepackage[utf8]{inputenc}%(only for the pdftex engine)

\usepackage{microtype}
\usepackage{amsmath,amssymb}
\usepackage{subcaption}

\newcommand{\bsk}{{\boldsymbol{k}}}

\newcommand{\bss}{{\boldsymbol{s}}}

\newcommand{\bsw}{{\boldsymbol{w}}}

\newcommand{\bslambda}{{\boldsymbol{\lambda}}}

\newcommand{\bsGamma}{{\boldsymbol{\Gamma}}}

\newcommand{\bbK}{{\mathbb{K}}}

\newcommand{\bbN}{{\mathbb{N}}}

\newcommand{\N}{{\mathbb{N}}} % natural numbers {1, 2, ...}
\newcommand{\R}{{\mathbb{R}}} % reals
\newcommand{\NN}{{\mathbb{N}}} % natural numbers {1, 2, ...}
\DeclareSymbolFont{bbold}{U}{bbold}{m}{n}
\DeclareSymbolFontAlphabet{\mathbbold}{bbold}

\newcommand{\calA}{{\mathcal{A}}}
\newcommand{\calB}{{\mathcal{B}}}
\newcommand{\calC}{{\mathcal{C}}}

\newcommand{\calF}{{\mathcal{F}}}
\newcommand{\calG}{{\mathcal{G}}}

\newcommand{\calJ}{{\mathcal{J}}}
\newcommand{\calK}{{\mathcal{K}}}

\newcommand{\calO}{{\mathcal{O}}}

\newcommand{\calS}{{\mathcal{S}}}

\newcommand{\calW}{{\mathcal{W}}}

\newcommand{\fraku}{{\mathfrak{u}}}

\providecommand{\argmin}{\operatorname*{argmin}}

\usepackage{color,mathtools,bbm,xspace}
\usepackage{algpseudocode}
\usepackage{algorithm, algorithmicx}
\usepackage[usenames,dvipsnames,svgnames,table]{xcolor}
\algnewcommand\algorithmicparam{\textbf{Parameters:}}
\algnewcommand\PARAM{\item[\algorithmicparam]}
\algnewcommand\algorithmicinput{\textbf{Input:}}
\algnewcommand\INPUT{\item[\algorithmicinput]}
\algnewcommand\RETURN{\State \textbf{Return }}

\newtheorem{theorem}{Theorem}
\newtheorem{lemma}{Lemma}
\newtheorem{corollary}{Corollary}
\newtheorem{proposition}{Proposition}

\newtheorem{remark}{Remark}
\newenvironment{proof}{\begin{trivlist}
\item[\hskip\labelsep{\it Proof.}]}{$\hfill\Box$\end{trivlist}}

\DeclareMathOperator{\SOL}{SOL}
\DeclareMathOperator{\APP}{APP}
\DeclareMathOperator{\ALG}{ALG}
\DeclareMathOperator{\ERR}{ERR}
\DeclareMathOperator{\card}{card}
\newcommand{\dataN}{\bigl(\hf(\bsk_i)\bigr)_{i=1}^n}
\newcommand{\dataNj}{\bigl(\hf(\bsk_i)\bigr)_{i=1}^{n_j}}
\newcommand{\dataNjd}{\bigl(\hf(\bsk_i)\bigr)_{i=1}^{n_{j^\dagger}}}
\newcommand{\ERRN}{\ERR\bigl(\dataN,n\bigr)}
\newcommand{\ERRNj}{\ERR\bigl(\dataNj,n_j\bigr)}
\newcommand{\ERRNjd}{\ERR\bigl(\dataNjd,n_{j^\dagger}\bigr)}
\DeclareMathOperator{\COST}{COST}
\DeclareMathOperator{\COMP}{COMP}
\newcommand{\hf}{\widehat{f}}
\newcommand{\hg}{\widehat{g}}
\newcommand{\tu}{\Tilde{u}}

\newcommand{\tcalK}{\widetilde{\calK}}

\newcommand{\lo}{\textup{lo}}
\newcommand{\up}{\textup{up}}
\newcommand{\inc}{\textup{in}}
\newcommand{\out}{\textup{out}}
\newcommand{\E}{\textup{e}}

\newcommand{\bigabs}[1]{\ensuremath{\bigl \lvert #1 \bigr \rvert}}

\newcommand{\biggabs}[1]{\ensuremath{\biggl \lvert #1 \biggr \rvert}}

\newcommand{\norm}[2][{}]{\ensuremath{\left \lVert #2 \right \rVert}_{#1}}

\newcommand{\bignorm}[2][{}]{\ensuremath{\bigl \lVert #2 \bigr \rVert}_{#1}}

\newcommand{\biggnorm}[2][{}]{\ensuremath{\biggl \lVert #2 \biggr \rVert}_{#1}}

\newcommand{\tLambda}{\widetilde{\Lambda}}
\newcommand{\bcalK}{\bar{\calK}}
\newcommand{\fapp}{f_{\text{app}}}

\allowdisplaybreaks

\begin{document}

 \title{Adaptive Approximation for Multivariate Linear Problems\\ with Inputs Lying in a Cone}
 \author{Yuhan Ding, Fred J. Hickernell, Peter Kritzer, Simon Mak}
 \maketitle
\begin{abstract}
\noindent We study adaptive approximation algorithms for general multivariate linear 
  problems where the sets of input functions are non-convex cones. While it is known that adaptive algorithms perform 
  essentially no better than non-adaptive algorithms for convex input sets, the 
  situation may be different for non-convex sets. A typical example
  considered here is function approximation based on series expansions. Given 
  an error tolerance, we use series coefficients of the input to construct an 
  approximate solution such that the error does not exceed this tolerance. We study the situation where we can bound the norm of the input based on a pilot sample, and the situation where we keep track of the decay rate of the series coefficients of the input. Moreover, 
  we consider situations where it makes sense to infer coordinate and smoothness importance. Besides performing an error analysis, we also study the information cost of our algorithms and the computational complexity 
  of our problems, and we identify conditions under which we can avoid a curse of 
  dimensionality.
\end{abstract}

%%%%%%%%%%%%%%%%%%%%%%%%%%%%
\section{Introduction} 
%%%%%%%%%%%%%%%%%%%%%%%%%%%%

In many situations, adaptive algorithms can be rigorously shown to perform \emph{essentially no better} than non-adaptive algorithms.  Yet, in practice adaptive algorithms are appreciated because they relieve the user from stipulating the computational effort required to achieve the desired accuracy.  The key to resolving this seeming contradiction is to construct a theory based on assumptions that favor adaptive algorithms. We do that here.

Adaptive algorithms infer the necessary computational effort based on the function data sampled.  Adaptive algorithms may perform better than non-adaptive algorithms if the set of input functions is \emph{non-convex}. We construct adaptive algorithms for general multivariate linear problems where the input functions lie in non-convex cones.  Our algorithms use a finite number of series coefficients of the input function to construct an approximate solution that satisfies an absolute error tolerance.  We show our algorithms to be essentially optimal.  We derive conditions under which the problem is tractable, i.e., the information cost of constructing the approximate solution does not increase exponentially with the dimension of the input function domain.  In the remainder of this section we define the problem and essential notation.  But first, we present a helpful example.

%%%%%%%%%%%%%%%%%%%%%%%%%%%%%%%%%%%%%%%%%%%%%%%%%%%%%%%%%%%%%%%%%%%%%%%%%%%%%%%%%%%%
\subsection{An Illustrative Example}\label{DHKM:secexamp}
%%%%%%%%%%%%%%%%%%%%%%%%%%%%%%%%%%%%%%%%%%%%%%%%%%%%%%%%%%%%%%%%%%%%%%%%%%%%%%%%%%%%

Consider the case of approximating functions defined over $[-1,1]^d$, using a Chebyshev polynomial basis.  The input function is denoted $f$, and the solution is $\SOL(f) = f$.  In this case,
	\begin{align*}
	f &  = \sum_{\bsk \in \bbN_0^d} \widehat{f}(\bsk) u_\bsk =: \SOL(f), 
	\qquad \bsk = (k_1, \ldots, k_d) \in \N_0^d,\\
    u_\bsk & := \prod_{\ell =1}^d \tu_{k_\ell} , 
    \qquad \tu_{k}(x) := \cos( k \cos^{-1}(x)) \quad \forall k \in \N_0.
	\end{align*}
Approximating $f$ well by a finite sum requires knowing which terms in the infinite series for $f$ are more important.  Let $\calF$ denote a Hilbert space of input functions where the norm of $\calF$ is a $\bslambda$-weighted norm of the series coefficients:
\begin{equation*}
    \norm[\calF]{f} := \norm[2]{\left(\frac{\hf(\bsk)}{\lambda_{\bsk}}\right)_{\bsk \in \N_0^d}}, \qquad \text{where } \bslambda = \bigl( \lambda_{\bsk} \bigr)_{\bsk \in \N_0^d}, \quad \lambda_\bsk := \prod_{\substack{\ell =1\\ k_\ell > 0}}^d \frac{w_\ell}{k^r_\ell}, \quad r > 0.
\end{equation*}
The $w_\ell$ are non-negative \emph{coordinate weights}, which embody the assumption that $f$ may depend more strongly on coordinates with larger $w_\ell$ than those with smaller $w_\ell$.  The definition of the $\calF$-norm implies that an input function must have series coefficients that decay quickly enough as the degree of the polynomial increases.  Larger $r$ implies smoother input functions.

The ordering of the weights,
\begin{equation} \label{DHKM:lambda_order}
    \lambda_{\bsk_1} \ge \lambda_{\bsk_2} \ge \cdots >0,
\end{equation}
implies an ordering of the wavenumbers, $\bsk$.  It is natural to approximate the solution using the first $n$ series coefficients as follows:
\begin{equation*} %\label{DHKM:Ex_APP_def}
   \APP(f,n) := \sum_{i=1}^n \hf(\bsk_i) u_{\bsk_i} \qquad \forall f \in \calF, \ n \in \N.
\end{equation*}

Here, we assume that it is possible to sample the series coefficients of the input function.  This is a less restrictive assumption than being able to sample any linear functional, but it is more restrictive than only being able to sample function values.  An important future problem is to extend the theory in this chapter to the case where the only function data available are function values.

The error of this approximation in terms of the norm on the output space, $\calG$, can be expressed as
\begin{gather*}
    \norm[\calG]{\SOL(f) - \APP(f,n)} = \norm[2]{\bigl(\hf(\bsk_i) \bigr)_{i=n+1}^{\infty}}, \\
    \text{where}\qquad
    \norm[\calG]{\sum_{\bsk \in \N_0^d} \hg(\bsk) u_{\bsk}} : = \norm[2]{\bigl(\hg(\bsk) \bigr)_{\bsk \in \N_0^d}}.
\end{gather*}
If one has a fixed data budget, $n$, then $\APP(f,n)$ is the best answer.  

However, our goal is an algorithm, $\ALG(f,\varepsilon)$ that satisfies the error criterion
\begin{equation} \label{DHKM:err_crit}
    \norm[\calG]{\SOL(f) - \ALG(f,\varepsilon)} \le \varepsilon \qquad \forall \varepsilon > 0, \ f \in \calC,
\end{equation}
where  $\varepsilon$ is the error tolerance, and $\calC \subset \calF$ is the set of input functions for which $\ALG$ is successful.  This algorithm contains a rule for choosing $n$---depending on $f$ and $\varepsilon$---so that $\ALG(f,\varepsilon) = \APP(f,n)$.  The objectives of this chapter are to 
\begin{itemize}
    \item construct such a rule, 
    \item choose a set $\calC$ of input functions for which the rule is valid,  
    \item characterize the information cost of $\ALG$, 
    \item determine whether $\ALG$ has optimal information cost, and 
    \item understand the dependence of this cost on the number of input variables, $d$, as well as the error tolerance, $\varepsilon$.
\end{itemize}
We return to this example in Section \ref{DHKM:revisexamp} to discuss the answers to some of these questions.  We perform some numerical experiments for this example in Section \ref{DHKM:numexamp_sec}.

%%%%%%%%%%%%%%%%%%%%%%%%%%%%%%%%%%%%%%%%%%%%%%%%%%%%%%%%%%%%%%%%%%%%%%%%%%%%%%%%%%%%
\subsection{General Linear Problem}
%%%%%%%%%%%%%%%%%%%%%%%%%%%%%%%%%%%%%%%%%%%%%%%%%%%%%%%%%%%%%%%%%%%%%%%%%%%%%%%%%%%%
Now, we define our problem more generally.  A solution operator maps the input function to an output, $\SOL:\calF \to \calG$.  As in the illustrative example above, the Banach spaces of inputs and outputs are defined by series expansions:
\begin{gather*}
    \calF := \left \{f = \sum_{\bsk \in \mathbb{K}} \hf(\bsk) u_{\bsk} : \norm[\calF]{f} : = \norm[\rho]{\left( \frac{\hf(\bsk)}{\lambda_{\bsk}} \right)_{\bsk \in \mathbb{K}}} < \infty \right\}, \quad 1 \le \rho \le \infty, \\
    \calG := \left \{g = \sum_{\bsk \in \mathbb{K}} \hg(\bsk) v_{\bsk} : \norm[\calG]{g} : = \norm[\tau]{\bigl(  \hg(\bsk)  \bigr)_{\bsk \in \mathbb{K}}} < \infty \right\}, \quad 1 \le \tau \le \rho.
\end{gather*}
Here, $\{u_{\bsk}\}_{\bsk \in \mathbb{K}}$ is a basis for the input Banach space $\calF$, $\{v_{\bsk}\}_{\bsk \in \mathbb{K}}$ is a basis for the output Banach space $\calG$, $\mathbb{K}$ is a countable index set, and $\bslambda = (\lambda_\bsk)_{\bsk \in \mathbb{K}}$ is the sequence of weights. These bases are defined to match the solution operator:
\begin{equation} \label{DHKM:basis_relate}
    \SOL(u_{\bsk}) = v_{\bsk} \qquad \forall \bsk \in \mathbb{K}.
\end{equation}
The $\lambda_{\bsk}$ represent the importance of the series coefficients of the input function.  The larger $\lambda_{\bsk}$ is, the more important $\hf(\bsk)$ is.

Although this problem formulation is quite general in some aspects, condition \eqref{DHKM:basis_relate} is somewhat restrictive.  In principle, the choice of basis can be made via the singular value decomposition, but in practice, if the norms of $\calF$ and $\calG$ are specified without reference to their respective bases, it may be difficult to identify bases satisfying \eqref{DHKM:basis_relate}.

To facilitate our derivations below, we establish the following lemma via H\"older's inequality:

\begin{lemma} \label{DHKM:Key_Lem}
Let $\calK$ be some proper or improper subset of the index set $\bbK$. Moreover, let $\rho'$ be defined by the relation
\begin{equation*}
    \frac 1\rho + \frac 1 {\rho'} = \frac 1 \tau, \qquad \text{i.e., } \rho' := \frac{\rho \tau}{\rho - \tau},
\end{equation*}
so $\tau \le \rho' \le \infty$.  Let $\Lambda :=  \bignorm[\rho']{\bigl(  \lambda_{\bsk}  \bigr)_{\bsk \in \calK}}$ be the norm of a subset of the weights.  Then the following are true for $f = \sum_{\bsk \in \calK} \hf(\bsk) u_{\bsk}$:
\begin{equation}
\label{DHKM:SOL_ineq}
    \norm[\calG]{\SOL(f)} = \norm[\tau]{\bigl(\hf(\bsk) \bigr)_{\bsk \in \calK}} \le \norm[\calF]{f} \, \Lambda,
    \end{equation}
    \begin{multline}
    \label{DHKM:SOL_tight_ineq}
    \bigabs{\hf(\bsk)} = \begin{cases}
    \displaystyle 
    \frac{R \lambda_{\bsk}^{\rho'/\rho + 1}}{\Lambda^{\rho'/\rho}}, & \forall \bsk \in \calK, \quad  \text{if}\ \rho'<\infty, \\
    R \Lambda \delta_{\bsk,\bsk^*}, & \forall \bsk \in \calK, \ \bsk^* \in \calK \text{ satisfies } \lambda_{\bsk^*} = \Lambda, \quad\text{if}\ \rho' = \infty,
    \end{cases}
   \\ 
    \implies  \ \ \norm[\calF]{f} = R \ \mbox{and} \ \norm[\calG]{\SOL(f)} = R \Lambda.
    \end{multline}
Equality \eqref{DHKM:SOL_tight_ineq} illustrates how inequality \eqref{DHKM:SOL_ineq} may be made tight.
\end{lemma}
\begin{proof}
We give the proof for $\rho' < \infty$.  The proof for $\rho' = \infty$ follows similarly. 
The proof of inequality \eqref{DHKM:SOL_ineq} proceeds by applying H\"older's inequality:  
\begin{align}
    \label{DHKM:SOL_A}
    \norm[\calG]{\SOL(f)}  
    & = \biggnorm[\calG]{\sum_{\bsk \in \calK} \hf(\bsk) v_{\bsk}} = \norm[\tau]{\bigl(  \hf(\bsk)  \bigr)_{\bsk \in \calK}}  = \left [\sum_{\bsk \in \calK}  \left\lvert\frac{\hf(\bsk)}{\lambda_{\bsk}} \right\rvert^{\tau} \lambda_{\bsk}^{\tau} \right]^{1/\tau} \\
    \nonumber
    & \le \biggnorm[\rho]{\biggl(  \frac{\hf(\bsk)}{\lambda_{\bsk}}  \biggr)_{\bsk \in \calK}} \, \bignorm[\rho']{\bigl(  \lambda_{\bsk}  \bigr)_{\bsk \in \calK}} = \norm[\calF]{f} \, \Lambda \qquad \text{since }\frac 1\rho + \frac 1 {\rho'} = \frac 1 \tau.
\end{align}
Substituting the formula for $\bigabs{\hf(\bsk)}$ in \eqref{DHKM:SOL_tight_ineq} into equation \eqref{DHKM:SOL_A} and applying the relationship between $\rho$, $\rho'$, and $\tau$ yields
\begin{equation*}
       \norm[\calG]{\SOL(f)}  
    =  \frac{R \bignorm[\tau]{\bigl(  \lambda_{\bsk}^{\rho'/\rho + 1}  \bigr)_{\bsk \in \calK}}} {\Lambda^{\rho'/\rho}} 
    = \frac{R \bignorm[\rho']{\bigl(  \lambda_{\bsk}  \bigr)_{\bsk \in \calK}}^{\rho'/\rho + 1}}
    {\Lambda^{\rho'/\rho}} = R \Lambda.
\end{equation*}
Moreover,
\begin{equation*}
    \norm[\calF]{f}  
    = \norm[\rho]{\left( \frac{\hf(\bsk)}{\lambda_{\bsk}} \right)_{\bsk \in \calK}}
    = \frac{R \bignorm[\rho]{\bigl(  \lambda_{\bsk}^{\rho'/\rho}  \bigr)_{\bsk \in \calK}}}{\Lambda^{\rho'/\rho}} 
    = \frac{R \bignorm[\rho']{\bigl(  \lambda_{\bsk}  \bigr)_{\bsk \in \calK}}^{\rho'/\rho}}
    {\Lambda^{\rho'/\rho}} = R.
\end{equation*}
This completes the proof.
\end{proof} \

Taking $\calK = \bbK$ in the lemma above, the norm of the solution operator can be expressed in terms of the norm of $\bslambda$ as follows:
\begin{equation} \label{DHKM:SOLNorm}
    \norm[\calF \to \calG]{\SOL}  = \sup_{\norm[\calF]{f} \le 1} \norm[\calG]{\SOL(f)} = \bignorm[\rho']{\bslambda}.
\end{equation}
We assume throughout this chapter that the weights are chosen to keep this norm is finite, namely,
\begin{equation} \label{DHKM:SOLNormFinite}
    \bignorm[\rho']{\bslambda} < \infty.
\end{equation}

As in Section \ref{DHKM:secexamp}, here in the general case the $\lambda_{\bsk}$ are assumed to have a known order as was specified in \eqref{DHKM:lambda_order}.
We also assume that all $\lambda_{\bsk}$ are  positive to avoid the trivial case where $\SOL(f)$ can be expressed exactly as a finite sum for all $f \in \calF$.

%%%%%%%%%%%%%%%%%%%%%%%%%%%%%%%%%%%%%%%%%
\subsection{An Approximation and an Algorithm}
%%%%%%%%%%%%%%%%%%%%%%%%%%%%%%%%%%%%%%%%%

The optimal approximation  based on $n$ series coefficients of the input function is defined in terms of the series coefficients of the input function corresponding to the largest $\lambda_{\bsk}$ as follows:
\begin{equation} \label{DHKM:APP_def}
    \APP : \calF \times \N_0 \to \calG, \quad  \APP(f,0) = 0, \ \ \APP(f,n) := \sum_{i=1}^n \hf(\bsk_i) v_{\bsk_i} \ \forall n \in \N.
\end{equation}
By the argument leading to \eqref{DHKM:SOL_A} it follows that 
\begin{equation} \label{DHKM:APP_Err_Coef}
    \norm[\calG]{\SOL(f) - \APP(f,n)} = \norm[\tau]{\bigl(\hf(\bsk_i)\bigr)_{i=n+1}^\infty}.
\end{equation}
An upper bound on the approximation error follows from Lemma \ref{DHKM:Key_Lem}:
\begin{equation} \label{DHKM:Refined_APP_err}
    \norm[\calG]{\SOL(f) - \APP(f,n) } \le \norm[\rho]{\left(\frac{\hf(\bsk_i)}{\lambda_{\bsk_i}}\right)_{i=n+1}^\infty}
    \bignorm[\rho']{\bigl(  \lambda_{\bsk_i}  \bigr)_{i = n+1}^{\infty}}.
\end{equation}
This leads to the following theorem.

\begin{theorem} \label{DHKM:APP_optimality_thm} Let $\calB_{R} : = \{ f \in \calF : \norm[\calF]{f} \le R \}$ denote the ball of radius $R$ in the space of input functions.  The error of the approximation defined in \eqref{DHKM:APP_def} is bounded tightly above as 
\begin{equation} \label{DHKM:APP_errorBd}
    \sup_{f \in \calB_R} \norm[\calG]{\SOL(f) - \APP(f,n)}  \le R \, \bignorm[\rho']{\bigl(  \lambda_{\bsk_i}  \bigr)_{i = n+1}^{\infty}}.
\end{equation}
Moreover, the worst case error over $\calB_R$ of $\APP'(\cdot,n)$, for any approximation based on $n$ series coefficients of the input function, can be no smaller.
\end{theorem}

\begin{proof}
The proof of \eqref{DHKM:APP_errorBd} follows immediately from \eqref{DHKM:Refined_APP_err} and  Lemma \ref{DHKM:Key_Lem}.  The optimality of $\APP$ follows by bounding the error of an arbitrary approximation, $\APP'$, applied to functions that mimic the zero function.

 Let $\APP'(0,n)$ depend on the series coefficients indexed by $\calJ  = \{\bsk'_1, \ldots, \bsk'_n\}$.  Use Lemma \ref{DHKM:Key_Lem} with $\calK = \bbK \setminus \calJ$ to choose $f$ to mimic the zero function, have norm $R$, and have as large a solution as possible, i.e.,
\begin{gather}
\nonumber
    \hf(\bsk'_1) = \cdots = \hf(\bsk'_n) = 0, \qquad \norm[\calF]{f} = R, \\ 
    \norm[\calG]{\SOL(f)} =  R \norm[\rho']{\left( \lambda_{\bsk} \right)_{\bsk \notin \calJ}} \qquad  \text{by \eqref{DHKM:SOL_tight_ineq}}.  \label{DHKM:FoolFun}
\end{gather}
Then $\APP'(\pm f,n) = \APP'(0,n)$ because $f$ mimics the zero function, and
\begin{align*}
\MoveEqLeft{\sup_{f \in \calB_R} \norm[\calG]{\SOL(f) - \APP(f,n)}} \\
& \ge \max_{\pm} \norm[\calG]{\SOL(\pm f) - \APP'(\pm f,n)} =  \max_{\pm} \norm[\calG]{\SOL(\pm f) - \APP'(0,n)} \\
& \ge \frac 12 \left [ \norm[\calG]{\SOL(f) - \APP'(0,n)} 
+ \norm[\calG]{- \SOL(f) - \APP'(0,n)}\right] \\
& \ge \norm[\calG]{\SOL(f)} 
 = R  \norm[\rho']{\left( \lambda_{\bsk} \right)_{\bsk \notin \calJ}} \qquad \text{by \eqref{DHKM:FoolFun}}.
\end{align*}
The ordering of the $\lambda_{\bsk}$ implies that $\norm[\rho']{\left( \lambda_{\bsk} \right)_{\bsk \notin \calJ}}$ for arbitrary $\calJ$ can be no smaller than the case $\calJ = \{\bsk_1, \ldots, \bsk_n\}$.  This completes the proof.
\end{proof} \

While approximation $\APP$ is a key piece of the puzzle, our ultimate goal is an algorithm, $\ALG : \calC \times [0,\infty)$, satisfying the absolute error criterion \eqref{DHKM:err_crit}. The non-adaptive Algorithm \ref{DHKM:BallAlg} satisfies this error criterion for $\calC  = \calB_R$.  

\begin{algorithm}[H]
\caption{Non-Adaptive $\ALG$ for a Ball of Input Functions \label{DHKM:BallAlg}}
	\begin{algorithmic}
	\PARAM the Banach spaces $\calF$ and $\calG$, including the weights $\bslambda$; the ball radius, $R$; $\APP$ satisfying \eqref{DHKM:APP_errorBd}
	\INPUT a black-box function, $f$; an absolute error tolerance, $\varepsilon>0$

    \Ensure Error criterion \eqref{DHKM:err_crit} for $\calC = \calB_{R}$

    \State Choose $n^* =  \min \left \{n \in \N_0 : \bignorm[\rho']{\bigl(  \lambda_{\bsk_i}  \bigr)_{i = n+1}^{\infty}} \le \varepsilon /R \right \}$

    \RETURN $\ALG(f,\varepsilon) = \APP(f,n^*)$
\end{algorithmic}
\end{algorithm}

After defining the information cost of an algorithm and the problem complexity in the next subsection, we demonstrate that this non-adaptive algorithm is optimal when the set of inputs is chosen to be $\calC = \calB_R$. However, typically one cannot bound the norm of the input function a priori, so Algorithm \ref{DHKM:BallAlg} is impractical. 

The key difficulty is that error bound \eqref{DHKM:APP_errorBd} depends on the norm of the input function.  In contrast, we will construct  error bounds for $\APP(f,n)$ that only depend on function data.  These will lead to \emph{adaptive} algorithms $\ALG$ satisfying error criterion \eqref{DHKM:err_crit}.  For such algorithms, the set of allowable input functions, $\calC$, will be a \emph{cone}, not a ball.

Note that algorithms satisfying error criterion \eqref{DHKM:err_crit} cannot exist for $\calC = \calF$. Any algorithm must require a finite sample size, even if it is huge.  Then, there must exist some $f \in \calF$ that looks exactly like the zero function to the algorithm but for which $\norm[\calG]{\SOL(f)}$ is arbitrarily large.  Thus, algorithms satisfying the error criterion  exist only for some strict subset of $\calF$.  Choosing that subset well is both an art and a science.

%%%%%%%%%%%%%%%%%%%%%%%%%%%%
\subsection{Information Cost and Problem Complexity}
%%%%%%%%%%%%%%%%%%%%%%%%%%%%

The information cost of $\ALG(f,\varepsilon)$ is denoted $\COST(\ALG,f,\varepsilon)$ and defined as the number of function data---in our situation, series coefficients---required by $\ALG(f,\varepsilon)$.  For adaptive algorithms this cost varies with the input function $f$.  We also define the information cost of the algorithm in general, recognizing that it will tend to depend on $\norm[\calF]{f}$:
\begin{equation*}
    \COST(\ALG, \calC, \varepsilon,R) : = \max_{f \in \calC \cap \calB_{R}} \COST(\ALG,f,\varepsilon).
\end{equation*}
Note that while the cost depends on $\norm[\calF]{f}$, $\ALG(f,\varepsilon)$ has no knowledge of $f$ beyond the fact that it lies in $\calC$.  It is common for $\COST(\ALG, \calC, \varepsilon,R)$ to be $\calO(\varepsilon^{-p})$, or perhaps asymptotically $c\log(1 + \varepsilon^{-1})$.

Let $\calA(\calC)$ denote the set of all possible algorithms that may be constructed using series coefficients and that \emph{satisfy error criterion \eqref{DHKM:err_crit}}.  We define the \emph{computational complexity} of a problem as the information cost of the best algorithm:
\begin{equation*}
    \COMP(\calA(\calC), \varepsilon,R) := \min_{\ALG \in \calA(\calC)} \COST(\ALG, \calC, \varepsilon,R) .
\end{equation*}
These definitions follow the information-based complexity literature \cite{TraWer98, TraWasWoz88}.
We define an algorithm to be \emph{essentially optimal} if there exist some fixed positive $\omega$, $\varepsilon_{\max}$, and $R_{\min}$ for which
\begin{multline} \label{DHKM:EssentialOpt}
    \COST(\ALG, \calC, \varepsilon,R) \le \COMP(\calA(\calC), \omega \varepsilon,R) \\ \forall \, \varepsilon \in (0, \varepsilon_{\max}], \ R \in [R_{\min}, \infty).
\end{multline}
If the complexity of the problem is $\calO(\varepsilon^{-p})$, the cost of an essentially optimal algorithm is also $\calO(\varepsilon^{-p})$. If the complexity of the problem is asymptotically $c \log(1 + \varepsilon^{-1})$, then  the cost of an essentially optimal algorithm is also asymptotically $c \log(1 + \varepsilon^{-1})$. 
We will show that our adaptive algorithms presented in Sections \ref{DHKM:pilot_sec} and \ref{DHKM:tracking_sec}  are essentially optimal.

\begin{theorem}\label{DHKM:thm_cost_non_adapt}
The non-adaptive Algorithm \ref{DHKM:BallAlg} has an information cost for the set of input functions $\calC = \calB_R$ that is given by
\[
\COST(\ALG, \calB_R, \varepsilon,R') = \min \left \{n \in \N_0 : \bignorm[\rho']{\bigl(  \lambda_{\bsk_i}  \bigr)_{i = n+1}^{\infty}} \le \varepsilon /R \right \}.
\]
This algorithm is essentially optimal for the set of input functions $\calB_R$, namely,
\begin{multline*}
\COST(\ALG, \calB_R, \varepsilon,R') \le \COMP(\calA(\calB_R),\omega \varepsilon ,R') \\ \forall \, \varepsilon \in (0, \varepsilon_{\max}], \ R \in [R_{\min}, \infty),
\end{multline*}
where $\varepsilon_{\max}$ and $R_{\min}$ are arbitrary and fixed, and $\omega = R_{\min}/R$.
\end{theorem}

\begin{proof}  Fix positive $\varepsilon_{\max}$, $R_{\min}$, $R$, and $\omega$ as defined above.  For $0 < \varepsilon \le \varepsilon_{\max}$ and $R_{\min} \le R' \le R$, the information cost of non-adaptive Algorithm  \ref{DHKM:BallAlg} follows from its definition.  Let 
\[
n^*(\varepsilon,R) : = \COST(\ALG, \calB_R, \varepsilon,R').
\]
Construct an input function $f \in \calB_{R'}$ as in the proof of Theorem \ref{DHKM:APP_optimality_thm} with $\calJ = \{\bsk_1, \ldots, \bsk_{n^*(\omega \varepsilon,R')} \}$. By the argument in the proof of Theorem \ref{DHKM:APP_optimality_thm}, any algorithm in $\calA(\calB_{R'})$ that can approximate $\SOL(f)$ with an error no greater than $\omega \varepsilon$ must use at least $n^*(\omega \varepsilon,R')$ series coefficients.  Thus, 
\begin{align*}
\COST(\ALG, \calB_R, \varepsilon,R') & =  n^*(\varepsilon,R) 
= n^*(\varepsilon R'/R,R') \\
& \le n^*(\omega \varepsilon, R') \qquad \text{since } R'/R \ge \omega
\\
& \le \COMP(\calA(\calB_{R'}),\omega \varepsilon, R') \le  \COMP(\calA(\calB_{R}),\omega \varepsilon, R').
\end{align*}
Thus, Algorithm \ref{DHKM:BallAlg} is essentially optimal.
\end{proof} \

For Algorithm \ref{DHKM:BallAlg}, the information cost, $\COST(\ALG,\calB_R,\varepsilon, R)$, depends on the decay rate of the tail norm of the $\lambda_{\bsk_i}$.  This decay may be algebraic or exponential and also determines the problem complexity, $\COMP(\calA(\calB_R),\varepsilon, R)$, as a function of the error tolerance, $\varepsilon$.

This theorem illustrates how an essentially optimal algorithm for solving a problem for a ball of input functions, $\calC = \calB_R$, can be non-adaptive.  However, as alluded to above, we claim that it is impractical to know a priori which ball your input function lies in.  On the other hand, in the situations described below where $\calC$ is a cone, we will show that $\calA(\calC)$ actually contains only adaptive algorithms via the lemma below.  The proof of this lemma follows directly from the definition of non-adaptivity.

\begin{lemma} \label{DHKM:NoNonAdpatLem}
For a given set of input functions, $\calC$, if  $\calA(\calC)$ contains any non-adaptive algorithms, then for every $\varepsilon > 0$,
\begin{equation*}
    \COMP(\calA(\calC),\varepsilon) : = \sup_{R > 0} \COMP(\calA(\calC),\varepsilon, R) < \infty.
\end{equation*}
\end{lemma}

%%%%%%%%%%%%%%%%%%%%%%%%%%%%
\subsection{Tractability}\label{DHKM:secTractability}
%%%%%%%%%%%%%%%%%%%%%%%%%%%%

Besides understanding the dependence of $\COMP(\calA(\calC), \varepsilon,R)$ on $\varepsilon$, we also want to understand how $\COMP(\calA(\calC), \varepsilon,R)$ depends on the dimension of the domain of the input function.  Suppose that $f: \Omega^d \to \R$, for some $\Omega \subseteq \R$, and let $\calF_d$ denote the dependence of the input space on the dimension $d$.  
The set of functions for which our algorithms succeed, $\calC_d$, depends on the dimension, too.  Also, $\SOL$, $\APP$, $\COST$, and $\COMP$ depend implicitly on dimension, and this dependence is sometimes indicated explicitly by the subscript $d$.

\bigskip

Different dependencies of $\COMP(\calA(\calC_d), \varepsilon,R)$ on the dimension $d$ and the error tolerance $\varepsilon$ are formalized as different notions of tractability. Since 
the complexity is defined in terms of the best available algorithm, tractability is a property that is inherent to the problem, not to a particular algorithm. 
We define the following notions of tractability (for further information on tractability we refer to the trilogy 
\cite{NovWoz08a}, \cite{NovWoz10a}, \cite{NovWoz12a}). Note that in contrast to 
these references we explicitly include the dependence on $R$ in our definitions. This 
dependence is natural for cones $\calC$ and might be different if $\calC$ is not a cone. 

\begin{itemize}   
\item We say that the adaptive approximation problem is strongly polynomially tractable
if and only if there are non-negative $C$, $p$, $\varepsilon_{\max}$, and $R_{\min}$ such that   
$$   
\COMP(\calA(\calC_d), \varepsilon,R)\le C\,R^p\,\varepsilon^{-p} \qquad \forall d\in\NN,\ \varepsilon \in (0, \varepsilon_{\max}], \ R \in [R_{\min}, \infty).
$$   
The infimum of $p$ satisfying the bound above is denoted by $p^*$   
and is called the exponent of strong polynomial tractability.    
\newline \qquad   

\item    
We say that the problem is polynomially tractable
if and only if there are non-negative $C,p$, $q$,  $\varepsilon_{\max}$, and $R_{\min}$ such that   
$$   
\COMP(\calA(\calC_d), \varepsilon,R)\le C\,d^{\,q}\,R^p\,   
\varepsilon^{-p}\qquad \forall d\in\NN,\ \varepsilon \in (0, \varepsilon_{\max}], \ R \in [R_{\min}, \infty).  
$$   
\vskip 0.5pc

\item   
We say that the problem is weakly tractable iff    
$$   
\lim_{d+R\,\varepsilon^{-1}\to\infty}\   
\frac{\log\, \COMP(\calA(\calC_d), \varepsilon,R)}   
{d+R\,\varepsilon^{-1}}\,=\,0.   
$$    
\end{itemize}   

Necessary and sufficient conditions on these tractability notions will be studied 
for different types of algorithms in Sections \ref{DHKM:SecPilotTract} and \ref{DHKM:SecDecayTract}. 

We remark that, for the sake of brevity, we focus here on tractability notions that are summarized as 
algebraic tractability in the recent literature 
(see, e.g., \cite{KriWoz19}). Theoretically, one could also study exponential tractability, where 
one would essentially replace $\varepsilon^{-1}$ 
by $\log (1 + \varepsilon^{-1})$ in the previous tractability notions. A more detailed study of 
tractability will be done in a future paper.

%%%%%%%%%%%%%%%%%%%%%%%%%%
\subsection{The Illustrative Example Revisited}\label{DHKM:revisexamp}
%%%%%%%%%%%%%%%%%%%%%%%%%%%
The example in Section \ref{DHKM:secexamp} chooses $\rho = \tau = 2$ and $\rho' = \infty$.  Thus, we obtain by Theorem \ref{DHKM:thm_cost_non_adapt}:
\begin{align*}
	\COMP(\calA(\calB_{R}),\varepsilon,R) &= \COST(\ALG,\calB_{R},\varepsilon,R) 
	=\min\{ n \in \N_0 : \lambda_{\bsk_{n+1}} \le \varepsilon/R \}.
\end{align*}
Using the non-increasing ordering of the $\lambda_{\bsk_i}$, we employ a standard technique for bounding the $n+1^\text{st}$ largest $\lambda_{\bsk}$ in terms of the sum of the $p^\text{th}$ power of all the $\lambda_{\bsk}$.  For $0  < 1/r < p$, 
\begin{align*}
(n+1)\lambda_{\bsk_{n+1}}^{p} &\le \sum_{i=1}^{n+1}\lambda_{\bsk_{i}}^{p} \le \sum_{\bsk \in \bbN_0^d}\lambda_{\bsk}^{p} = \prod_{\ell =1}^d \left[1 + w_\ell^{p}\sum_{k = 1}^\infty \frac {1}{k^{pr}} \right ] \\
& =  \prod_{\ell =1}^d \left[1 + w_\ell^{p}\zeta(pr)\right ] 
= \exp\left(\sum_{\ell = 1}^d \log\bigl(1 + w_\ell^{p}\zeta(pr) \bigr) \right) \\
& \le \exp\left(\zeta(pr) \sum_{\ell = 1}^{\infty} w_\ell^{p} \right) \qquad \text{since }\log(1+x) \le x \ \forall x\ge 0.
\end{align*}
Hence, substituting the above upper bound on $\lambda_{\bsk_{n+1}}$ into the formula for the complexity of the problem, we obtain an upper bound on the complexity:
\begin{align*}
    \MoveEqLeft{\COMP(\calA(\calB_{R}),\varepsilon,R)} \\
  &\le\min\left\{ n \in \N_0 : \frac{1}{n+1}\exp\left(\zeta(pr) \sum_{\ell = 1}^\infty w_\ell^{p} \right)  \le \left( \frac{\varepsilon}{R} \right)^{p}  \right\} \\
  &= \left \lceil \left( \frac{R}{\varepsilon} \right)^{p}  \exp\left(\zeta(pr) \sum_{\ell = 1}^{\infty} w_\ell^{p} \right) \right \rceil - 1.
\end{align*}

If $p^\dagger$ is the infimum of the $p$ for which $\sum_{\ell = 1}^{\infty} w_\ell^{p}$ is finite, and $p^\dagger$ is finite, then we obtain strong polynomial tractability and an exponent of strong tractability that is $p^* = \max(1/r,p^\dagger)$. On the other hand, if the coordinate weights are all unity,  $w_1 = w_2 = \cdots = 1$, then there are $2^d$ different $\lambda_{\bsk}$ with a value of $1$, and so $\COMP(\calA(\calB_{R}),\varepsilon,R) \ge 2^d$, and the problem is not tractable.

\subsection{What Comes Next}
In the following section we define a cone of input functions, $\calC$, in \eqref{DHKM:pilot_cone} whose norms can be bounded above in terms of the series coefficients obtained from a pilot sample.  Adaptive Algorithm \ref{DHKM:PilotConeAlg} is shown to be optimal for this $\calC$.  We also identify necessary and sufficient conditions for tractability.

Section \ref{DHKM:tracking_sec} considers the situation where function data is relatively inexpensive, and we track the decay rate of the series coefficients.  Adaptive Algorithm \ref{DHKM:TrackConeAlg} is shown to be optimal in this situation.

Section \ref{DHKM:smoothimportance_sec} considers the case where the most suitable weights $\bslambda$ are not known a priori and are instead inferred from function data.  Adaptive Algorithm \ref{DHKM:InfPilotConeAlg} combines this inference step with Algorithm \ref{DHKM:PilotConeAlg} to construct an approximation that satisfies the error criterion.

%%%%%%%%%%%%%%%%%%%%%%%%%%%%%%%%%%%%%%%%%%%%%%%%%%%%%%%%%%%%%%%%%%%%%%%%%%%%%
%%%%%%%%%%%%%%%%%%%%%%%%%%%%%%%%%%%%%%%%%%%%%%%%%%%%%%%%%%%%%%%%%%%%%%%%%%%%%
%%%%%%%%%%%%%%%%%%%%%%%%%%%%%%%%%%%%%%%%%%%%%%%%%%%%%%%%%%%%%%%%%%%%%%%%%%%%%
%%%%%%%%%%%%%%%%%%%%%%%%%%%%%%%%%%%%%%%%%%%%%%%%%%%%%%%%%%%%%%%%%%%%%%%%%%%%%
\section{Bounding the Norm of the Input Function Based on a Pilot Sample} 
\label{DHKM:pilot_sec} 
%%%%%%%%%%%%%%%%%%%%%%%%%%%%%%%%%%%%%%%%%%%%%%%%%%%%%%%%%%%%%%%%%%%%%%%%%%%%%
%%%%%%%%%%%%%%%%%%%%%%%%%%%%%%%%%%%%%%%%%%%%%%%%%%%%%%%%%%%%%%%%%%%%%%%%%%%%%
%%%%%%%%%%%%%%%%%%%%%%%%%%%%%%%%%%%%%%%%%%%%%%%%%%%%%%%%%%%%%%%%%%%%%%%%%%%%%

%%%%%%%%%%%%%%%%%%%%%%%%%%%%
\subsection{The Cone and the Optimal Algorithm}
%%%%%%%%%%%%%%%%%%%%%%%%%%%%

The premise of an adaptive algorithm is that the finite information we observe about the input function 
tells us something about what is not observed.  Let $n_1$ denote the number of pilot observations, based on the set of wavenumbers
\begin{equation} \label{DHKM:KOnedef}
    \calK_1 := \{\bsk_1, \ldots, \bsk_{n_1} \},
\end{equation}
where the $\bsk_i$ are defined by the ordering of the $\lambda_{\bsk}$ in \eqref{DHKM:lambda_order}.  Let $A$ be some constant inflation factor greater than one.  The cone of functions whose norm can be bounded well in terms of a pilot sample, $\{\hf(\bsk_1), \ldots, \hf(\bsk_{n_1})\}$,
is given by
\begin{equation} \label{DHKM:pilot_cone}
    \calC = \left \{ f \in \calF : \norm[\calF]{f} \le A \norm[\rho]{\left( \frac{\hf(\bsk)}{\lambda_{\bsk}} \right)_{\bsk \in \calK_1}} \right\}.
\end{equation}

Referring to error bound \eqref{DHKM:Refined_APP_err}, we see that the error of $\APP(f,n)$ depends on the series coefficients not sampled.  The definition of $\calC$ allows us to bound these as follows: 
\begin{align*}
     \norm[\rho]{\left(\frac{\hf(\bsk_i)}{\lambda_{\bsk_i}}\right)_{i=n+1}^\infty} & =  \left[ \norm[\calF]{f}^\rho -  \norm[\rho]{\left(\frac{\hf(\bsk_i)}{\lambda_{\bsk_i}}\right)_{i=1}^n}^\rho
    \right]^{1/\rho} \qquad \forall f \in \calF \\
    &  \le  \left[ A^\rho \norm[\rho]{\left( \frac{\hf(\bsk)}{\lambda_{\bsk}} \right)_{\bsk \in \calK_1}}^\rho -  \norm[\rho]{\left(\frac{\hf(\bsk_i)}{\lambda_{\bsk_i}}\right)_{i=1}^n}^\rho
    \right]^{1/\rho} \quad \forall f \in \calC.
\end{align*}
This inequality together with error bound \eqref{DHKM:Refined_APP_err} implies the data-based error bound 
\begin{subequations} \label{DHKM:pilot_errbd}
\begin{equation}
\norm[\calG]{\SOL(f) - \APP(f,n)}  \le \ERRN  \qquad \forall f \in \calC,
\end{equation}
where 
\begin{multline}
\ERRN \\
    : =  
    \left[ A^\rho \norm[\rho]{\left( \frac{\hf(\bsk)}{\lambda_{\bsk}} \right)_{\bsk \in \calK_1}}^\rho -  \norm[\rho]{\left(\frac{\hf(\bsk_i)}{\lambda_{\bsk_i}}\right)_{i=1}^n}^\rho \right]^{1/\rho} 
    \, \bignorm[\rho']{\bigl(  \lambda_{\bsk_i}  \bigr)_{i = n+1}^{\infty}} , 
    \\ n \ge n_1.
\end{multline}
\end{subequations}
This error bound decays as $\bignorm[\rho]{\bigl( f(\bsk_i)/\lambda_{\bsk_i} \bigr)_{i=1}^n}$ increases and as the tail norm of the $\lambda_{\bsk_i}$ decreases.  This data-driven error bound underlies  Algorithm \ref{DHKM:PilotConeAlg}, which is successful for $\calC$ defined in \eqref{DHKM:pilot_cone}:

\begin{algorithm}
	\caption{$\ALG$ Based on a Pilot Sample\label{DHKM:PilotConeAlg}} 
	\begin{algorithmic}
	\PARAM the Banach spaces $\calF$ and $\calG$, including the weights $\bslambda$; an initial sample size, $n_1 \in \N$; an inflation factor, $A > 1$; $\APP$ satisfying \eqref{DHKM:Refined_APP_err}
		\INPUT a black-box function, $f$; an absolute error tolerance,
		$\varepsilon>0$

\Ensure Error criterion \eqref{DHKM:err_crit} for  the cone defined in \eqref{DHKM:pilot_cone}

\State Let $n \leftarrow n_1 -1$
\Repeat

\State Let $n \leftarrow n + 1$

\State Compute $\ERRN$ as defined in \eqref{DHKM:pilot_errbd}

\Until $\ERRN \le \varepsilon$

\RETURN $\ALG(f,\varepsilon) = \APP(f,n)$

\end{algorithmic}
\end{algorithm}

\begin{theorem} \label{DHKM:PilotCostThm}
Algorithm \ref{DHKM:PilotConeAlg} yields an answer satisfying absolute error criterion \eqref{DHKM:err_crit}, i.e., $\ALG \in \calA(\calC)$ for $\calC$ defined in \eqref{DHKM:pilot_cone}.  The information cost is
\begin{multline} \label{DHKM:PilotConeAlg_cost}
    \COST(\ALG,\calC,\varepsilon,R) \\
    = \min \left \{n \ge n_1 : \bignorm[\rho']{\bigl(  \lambda_{\bsk_i}  \bigr)_{i = n+1}^{\infty}} \,
    \le \varepsilon/[(A^\rho -1)^{1/\rho}R] \right \}.
\end{multline}
There exist positive $\varepsilon_{\max}$ and $R_{\min}$ for which the computational complexity has the lower bound
\begin{multline} \label{DHKM:PilotConeAlg_comp}
        \COMP(\calA(\calC),\varepsilon,R) \ge \min \left \{n \ge n_1 : \bignorm[\rho']{\bigl(  \lambda_{\bsk_i}  \bigr)_{i = n+1}^{\infty}} \,
    \le 2\varepsilon/[(1 - 1/A) R] \right \} \\
    \forall \varepsilon \in (0, \varepsilon_{\max}], \ R \in [R_{\min}, \infty).
\end{multline}
Algorithm \ref{DHKM:PilotConeAlg} is essentially optimal.  Moreover, $\calA(\calC)$ contains only adaptive algorithms.
\end{theorem}

\begin{proof} 
The upper bound on the computational cost of this algorithm is obtained by noting that 
\begin{align*}
    \MoveEqLeft[1]{\COST(\ALG,\calC,\varepsilon,R)} \\
    & = \max_{f \in \calC \cap \calB_{R}} \min \left \{n \ge n_1 : \ERRN \le \varepsilon \right \} \\
     & \le \max_{f \in \calC \cap \calB_{R}} \min \left \{n \ge n_1 : 
     (A^\rho -1)^{1/\rho} \norm[\rho]{\left( \frac{\hf(\bsk)}{\lambda_{\bsk}} \right)_{\bsk \in \calK_1 }} \, 
     \bignorm[\rho']{\bigl(  \lambda_{\bsk_i}  \bigr)_{i = n+1}^{\infty}}
    \le \varepsilon \right \} \\   
     & \le \min \left \{n \ge n_1 : 
     (A^\rho -1)^{1/\rho} R \bignorm[\rho']{\bigl(  \lambda_{\bsk_i}  \bigr)_{i = n+1}^{\infty}} 
    \le \varepsilon \right \},  
\end{align*}
since $\bignorm[\rho]{\bigl( \hf(\bsk)/\lambda_{\bsk} \bigr)_{\bsk \in \calK_1}} \le \bignorm[\rho]{\bigl( \hf(\bsk_i)/\lambda_{\bsk_i} \bigr)_{i=1}^{n}} \le  \norm[\calF]{f} \le R$ for all $f \in \calB_R$, $n \ge n_1$.  Moreover, this inequality is tight for some $f \in \calC \cap \calB_{R}$, namely, those certain $f$ for which $\hf(\bsk_i) = 0$ for $i > n_1$.  This completes the proof of \eqref{DHKM:PilotConeAlg_cost}.

To prove the lower complexity bound, choose $\varepsilon_{\max}$ and $R_{\min}$ such that 
\[
\bignorm[\rho']{\bigl(  \lambda_{\bsk_i}  \bigr)_{i = n_1+1}^{\infty}} \,
    > 2\varepsilon_{\max}/[(1 - 1/A) R_{\min}].
    \]
Let $\ALG'$ be any algorithm that satisfies the error criterion, \eqref{DHKM:err_crit}, for this choice of $\calC$ in \eqref{DHKM:pilot_cone}.   Fix $R \in [R_{\min},\infty)$ and $\varepsilon \in (0,\varepsilon_{\max}]$ arbitrarily.  Two fooling functions will be constructed of the form $f_\pm = f_1 \pm f_2$.  

The input function $f_1$ is defined via its series coefficients as in Lemma \ref{DHKM:Key_Lem}, having nonzero coefficients only for $\bsk \in \calK_1$:
\begin{equation*}
    \bigabs{\hf_1(\bsk)} = \begin{cases} \displaystyle \frac{R (1+1/A) \lambda_{\bsk}^{\rho'/\rho + 1}}{2\bignorm[\rho']{\bigl(  \lambda_{\bsk}  \bigr)_{\bsk \in \calK_1}}^{\rho'/\rho}}, &  \bsk \in \calK_1, \\
    0, & \bsk \notin \calK_1,
    \end{cases}
   \qquad \norm[\calF]{f_1} = \frac{R(1 + 1/A)}{2}.
\end{equation*}
Suppose that $\ALG'(f_1,\varepsilon)$ samples the series coefficients $\hf_1(\bsk)$ for $\bsk \in \calJ$, and let $n$ denote the cardinality of $\calJ$.  

Now, construct the input function $f_2$, having zero coefficients for $\bsk \in \calJ$ and also as in Lemma \ref{DHKM:Key_Lem}:
\begin{gather}
\nonumber
    \bigabs{\hf_2(\bsk)} = \begin{cases} \displaystyle \frac{R (1-1/A) \lambda_{\bsk}^{\rho'/\rho + 1}}{2\bignorm[\rho']{\bigl(  \lambda_{\bsk}  \bigr)_{\bsk \notin \calJ}}^{\rho'/\rho}}, &  \bsk \notin \calJ, \\
    0, & \bsk \in \calJ, 
    \end{cases}
    \qquad \norm[\calF]{f_2} = \frac{R(1 - 1/A)}{2}, \\
    \label{DHKM:SOLf2bd}
    \norm[\calG]{\SOL(f_2)} = \frac{R(1 - 1/A)}{2} \, \bignorm[\rho']{\bigl(  \lambda_{\bsk}  \bigr)_{\bsk \notin \calJ}}.
\end{gather}
Let $f_{\pm} = f_1 \pm f_2$.  By the definitions above, it follows that
\begin{align}
\nonumber
    \norm[\calF]{f_{\pm}} &= \norm[\calF]{ f_1 \pm f_2 } \le \norm[\calF]{ f_1} + \norm[\calF]{ f_2 } =  R, \\
    \nonumber
    \norm[\rho]{\left( \frac{\hf_\pm(\bsk_i)}{\lambda_{\bsk_i}} \right)_{i=1}^{n_1}} 
    & = \norm[\rho]{\left( \frac{\hf_1(\bsk_i) \pm \hf_2(\bsk_i)}{\lambda_{\bsk_i}} \right)_{i=1}^{n_1}} \\
    \nonumber
    & \ge \norm[\rho]{\left( \frac{\hf_1(\bsk_i)}{\lambda_{\bsk_i}} \right)_{i=1}^{n_1}} - \norm[\rho]{\left( \frac{\hf_2(\bsk_i)}{\lambda_{\bsk_i}} \right)_{i=1}^{n_1}} \\
    \nonumber
    & \ge \norm[\calF]{ f_1} - \norm[\calF]{ f_2 } =  \frac{R}{A} \ge \frac{\norm[\calF]{f_{\pm}}}{A}.
\end{align}
Therefore, $f_\pm \in \calC \cap \calB_R$.  Moreover, since the series coefficients for $f_\pm$ are the same for $\bsk \in \calJ$, it follows that $\ALG'(f_+,\varepsilon) = \ALG'(f_-,\varepsilon)$.  Thus, $\SOL(f_{+})$ must be quite similar to $\SOL(f_{-})$.

Using an argument like that in the proof of  Theorem \ref{DHKM:APP_optimality_thm}, it follows that 
\begin{align*}
\varepsilon & \ge \max_{\pm} \norm[\calG]{\SOL(f_{\pm}) - \ALG'(f_{\pm},\varepsilon)} 
=  \max_{\pm} \norm[\calG]{\SOL(f_{\pm}) - \ALG'(f_{+},\varepsilon)} \\
& \ge \frac 12 \left [ \norm[\calG]{\SOL(f_{+}) - \ALG'(f_{+},\varepsilon)} 
+ \norm[\calG]{\SOL(f_{-}) - \ALG'(f_{+},\varepsilon)}  \right] \\
& \ge \frac 12 \norm[\calG]{\SOL(f_+ - f_-)} = \norm[\calG]{\SOL(f_2)} 
=\frac{R(1 - 1/A)}{2} \, \bignorm[\rho']{\bigl(  \lambda_{\bsk}  \bigr)_{\bsk \notin \calJ}} \qquad \text{by \eqref{DHKM:SOLf2bd}} \\
& \ge \frac{R(1 - 1/A)}{2} \, \bignorm[\rho']{\bigl(  \lambda_{\bsk_i}  \bigr)_{i = n+1}^\infty},
\end{align*}
by the ordering of the $\bsk$ in \eqref{DHKM:lambda_order}.  By the choice of $R_{\min}$ and $\varepsilon_{\max}$ above, it follows that $n > n_1$.  This inequality then implies lower complexity bound \eqref{DHKM:PilotConeAlg_comp}. Because $\lim_{R \to \infty} \COMP(\calA(\calC), \varepsilon,R) = \infty$  it follows from Lemma \ref{DHKM:NoNonAdpatLem} that $\calA(\calC)$ contains only adaptive algorithms.

The essential optimality of Algorithm \ref{DHKM:PilotConeAlg} follows by observing that 
\[
\COST(\ALG,\calC,\varepsilon,R) \le \COMP(\calA(\calC),\omega \varepsilon,R) \qquad \text{for } \omega = \frac{1-1/A}{2(A^\rho -1)^{1/\rho}}.
\]
This satisfies definition \eqref{DHKM:EssentialOpt}.  
\end{proof} \

The above derivation assumes that $A > 1$.  If $A =1$, then our cone consists of functions whose series coefficients vanish for wavenumbers outside $\calK_1$.  The exact solution can be constructed using only the pilot sample.  Our algorithm is then non-adaptive, but succeeds for input functions in the cone $\calC$, which is an unbounded set.

We may not be able to guarantee that a particular $f$ of interest lies in our cone, $\calC$, but we may derive necessary conditions for $f$ to lie in $\calC$.  The following proposition follows from the definition of $\calC$ in \eqref{DHKM:pilot_cone} and the fact that the term on the left below underestimates $\norm[\calF]{f}$.

\begin{proposition}
If $f \in \calC$, then 
\begin{equation} \label{DHKM:PilotConeNecessary}
    \norm[\rho]{\left( \frac{\hf(\bsk_i)}{\lambda_{\bsk_i}} \right)_{\bsk_i =1}^{n}} \le A
    \norm[\rho]{\left( \frac{\hf(\bsk)}{\lambda_{\bsk}} \right)_{\bsk \in \calK_1}} \qquad \forall n \in \NN.
\end{equation}
\end{proposition}

If condition \eqref{DHKM:PilotConeNecessary} is violated in practice, then $f \notin \calC$, and Algorithm \ref{DHKM:PilotConeAlg} may output an incorrect answer.  The remedy is to make $\calC$ more inclusive by increasing the inflation factor, $A$, and/or the pilot sample size, $n_1$.

%%%%%%%%%%%%%%%%%%%%%%%%%%%%
\subsection{Tractability}\label{DHKM:SecPilotTract}
%%%%%%%%%%%%%%%%%%%%%%%%%%%%

In this section, we write $\calC_d$ instead of $\calC$, to stress the dependence on $d$, and for the same reason we write $\lambda_{d,\bsk_i}$ instead of $\lambda_{\bsk_i}$. Recall that we assume that 
$\lambda_{d,\bsk_1}\ge \lambda_{d,\bsk_2}\ge \cdots >0$. Let 
\[
n(\delta,d) :=\min \left \{n\ge 0: \bignorm[\rho']{\bigl(  \lambda_{d,\bsk_i}  \bigr)_{i = n+1}^{\infty}} \,
    \le \delta \right \} \qquad \forall \delta > 0.
\]
From 
Equations \eqref{DHKM:PilotConeAlg_cost} and \eqref{DHKM:PilotConeAlg_comp}, we obtain that 
\begin{multline*}
     \COMP(\calA(\calC_d),\omega_{\textup{lo}} \varepsilon,R) \le n(\varepsilon/R,d) \le 
        \COMP(\calA(\calC_d),\omega_{\textup{hi}} \varepsilon,R)  \\
        \forall \varepsilon \in (0, \varepsilon_{\max}], \ R \in [R_{\min}, \infty),
\end{multline*}
where the positive constants $\omega_{\textup{lo}}$ and $\omega_{\textup{hi}}$  depend on $A$, but not depend on $d$, $\varepsilon$, or $R$.  From the equation above, it is clear that tractability depends on the behavior of $n(\varepsilon/R,d)$ as $R/\varepsilon$ and $d$ tend to infinity. We would like to study under which conditions we obtain the various tractability notions defined in Section \ref{DHKM:secTractability}. 

To this end, we distinguish two cases, depending on whether $\rho'$ is infinite or not. This 
distinction is useful because it allows us to relate the computational complexity of the algorithms 
considered in this chapter to the computational complexity of linear problems on certain function spaces considered in the classical literature on information-based complexity, as for example \cite{NovWoz08a}. The case $\rho'=\infty$ corresponds to the
worst-case setting, where one studies the worst performance of an algorithm over the unit ball of 
a space. The results in Theorem \ref{DHKM:thmtract1} below are indeed very similar to the results 
for the worst-case setting over balls of suitable function spaces. The case $\rho<\infty$ corresponds to the so-called average-case setting, where one 
considers the average performance over a function space equipped with a suitable measure. 
For both of these settings there exist tractability results that we will make use of here.

\paragraph*{CASE 1: $\rho'=\infty$:}

If $\rho'=\infty$, we have, due to the monotonicity of the $\lambda_{d,\bsk_i}$, 
\[
n(\varepsilon/R,d)=\min \left \{n\ge 0\colon \lambda_{d,\bsk_{n+1}} \,
    \le \varepsilon/R \right \}.
\]
We then have the following theorem.

\begin{theorem} \label{DHKM:thmtract1}
Using the same notation as above, the following statements hold for the case $\rho'=\infty$.
 \begin{itemize}
  \item[1.] 
  We have strong polynomial tractability if and only if there exist $\eta>0$ and $i_0\in\NN$ such that
 \begin{equation}\label{DHKM:eq:condspt}
    \sup_{d\in\NN} \sum_{i=i_0}^\infty \lambda_{d,\bsk_i}^\eta < \infty.
 \end{equation}
 Furthermore, the exponent of strong polynomial tractability is then equal to the infimum of those $\eta>0$ for which \eqref{DHKM:eq:condspt} holds. 
 \item[2.] 
  We have polynomial tractability if and only if there exist $\eta_1, \eta_2 \ge 0$ and $\eta_3, K>0$ such that
 \[
    \sup_{d\in\NN} d^{-\eta_1}\, \sum_{i=\lceil K d^{\eta_2} \rceil}^\infty \lambda_{d,\bsk_i}^{\eta_3} < \infty.
 \]
 \item[3.] 
 We have weak tractability if and only if 
 \begin{equation}\label{DHKM:eq:condwt}
  \sup_{d\in\NN} \, \exp(-cd) \sum_{i=1}^\infty \exp\left(-c\left(\frac{1}{\lambda_{d,\bsk_i}}\right)\right) <\infty\quad \mbox{for all}\quad c>0.
 \end{equation}
\end{itemize}
\end{theorem}

\begin{proof}
Letting $\widetilde{\varepsilon}:=\sqrt{\varepsilon/R}$, we see that 
$n(\varepsilon/R,d)=\min \left \{n\ge 0\colon \lambda_{d,\bsk_{n+1}} \,\le \widetilde{\varepsilon}^{\,2} \right \}$. The latter expression is well studied in the context of tractability of linear problems in the worst-case setting defined on unit balls of certain spaces, and if and only if conditions on the $\lambda_{d,\bsk_i}$ for various tractability notions are known. These conditions can be found in \cite[Chapter 5]{NovWoz08a} for (strong) polynomial tractability and  \cite{WerWoz17} for weak tractability.

Since, in this chapter, we consider $\min \left \{n\ge 0\colon \lambda_{d,\bsk_{n+1}} \,\le \varepsilon/R \right \}$, and in \cite{NovWoz08a} and \cite{WerWoz17} $\varepsilon /R$ is replaced by the square of the error tolerance,
there are slight differences between the results here and those in the aforementioned references; to be more precise, the exponent of strong polynomial tractability is $\eta$ here, whereas it is $2\eta$ in \cite{NovWoz08a}, and $1/\lambda_{d,\bsk_i}$ in \eqref{DHKM:eq:condwt} corresponds to $1/\sqrt{\lambda_{d,\bsk_i}}$ in \cite{WerWoz17}. 
\end{proof}

\paragraph*{CASE 2: $\rho'<\infty$:} 

In this case, letting $\widetilde{\varepsilon}:=(\varepsilon/R)^{\rho'/2}$ and 
$\widetilde{\lambda}_{d,i}:=\lambda_{d,\bsk_i}^{\rho'}$, we have 
\begin{align}
    \nonumber 
    n(\varepsilon/R,d) &=\min \left \{n\ge 0\colon 
\sum_{i=n+1}^\infty \lambda_{d,\bsk_i}^{\rho'}\,
    \le (\varepsilon/R)^{\rho'} \right \} \\
    \label{DHKM:eqaveragetract}
    & =\min \left \{n\ge 0\colon 
\sum_{i=n+1}^\infty \widetilde{\lambda}_{d,i}\,
    \le \widetilde{\varepsilon}^{\,2} \right \}.
\end{align}
However, the latter expression corresponds exactly to the 
average-case tractability (with respect to the parameters 
$\widetilde{\lambda}_{d,i}$ and $\widetilde{\varepsilon}$) defined 
on certain spaces as studied in, e.g., \cite{NovWoz08a}. 
This leads us to the following theorem.
\begin{theorem} \label{DHKM:thmtract2}
Using the same notation as above, the following statements hold for the case $\rho'<\infty$.
 \begin{itemize}
  \item[1.] 
  We have strong polynomial tractability if and only if there exist $\eta\in (0,1)$ and $i_0\in\NN$ such that
 \begin{equation}\label{DHKM:eq:condspt1}
    \sup_{d\in\NN} \sum_{i=i_0}^\infty \lambda_{d,\bsk_i}^{\rho'\,\eta} < \infty.
 \end{equation}
 Furthermore, the exponent of strong polynomial tractability is then 
 \[
 \inf\left\{\rho'\eta/(1-\eta)\colon \mbox{$\eta$ satisfies \eqref{DHKM:eq:condspt1}}\right\}.
 \]
 \item[2.] 
  We have polynomial tractability if and only if there exist $\eta_1, \eta_2 \ge 0$ and $\eta_3\in (0,1), K>0$ such that
 \[
    \sup_{d\in\NN} d^{-\eta_1}\, \sum_{i=\lceil K d^{\eta_2} \rceil}^\infty \lambda_{d,\bsk_i}^{\rho'\,\eta_3} < \infty.
 \]
 \item[3.] Let $t_{d,i}:=\sum_{k=i}^\infty \lambda_{d,\bsk_i}$.
 We have weak tractability if and only if 
 \[
   \lim_{i\to\infty} t_{d,i}\, (\log i)^2=0\quad\mbox{for all $d$},
 \]
 and there exists a function $f:[0,1/2)\to \{1,2,3,\ldots\}$ such that
\[
  \sup_{\beta\in (0,1/2]}\, \beta^{-2} \,
  \sup_{d\ge f(\beta)}\,\, \sup_{i\ge \lceil \exp (d\sqrt{b}) \rceil +1}\, \, \lim_{i\to\infty} t_{d,i}\, (\log i)^2
  < \infty.
\]
 \end{itemize}
\end{theorem}
\begin{proof}
  The proof of the theorem is similar to that of Theorem \ref{DHKM:thmtract1}, using \eqref{DHKM:eqaveragetract}.
\end{proof}
\begin{remark}
  Results for further tractability notions, such as quasi-polynomial tractability or $(s,t)$-weak 
  tractability, can be shown using similar arguments as above and results from \cite{KriWoz19}, \cite{NovWoz10a}, \cite{WerWoz17}, and the papers cited therein. 
\end{remark}

To be more concrete, we consider the situation where the $\lambda_{\bsk}$ are specified in terms of positive \emph{coordinate weights}, $w_1, \ldots, w_d$, and positive \emph{smoothness weights}, $s_1, s_2, \ldots$:
\begin{equation}
    \label{DHKM:prodwts}
\lambda_{d,\bsk}  := \prod_{\substack{\ell =1\\ k_\ell > 0}}^d w_\ell s_{k_\ell}, \qquad \bsk \in \NN_0^d, \ d \in \NN.
\end{equation}
This is a generalization of the example in Section \ref{DHKM:secexamp}, where $s_{k} = k^{-r}$.  This form of the $\lambda_{d,\bsk}$ is considered in greater detail in Section \ref{DHKM:smoothimportance_sec}.
The same argument as in Section \ref{DHKM:revisexamp} implies that the sum of the $\lambda_{d,\bsk_i}^\eta$ is bounded above as
\[
\sum_{i=1}^{\infty}\lambda_{d, \bsk_{i}}^{\eta} \le \exp\left(\sum_{k=1}^\infty s_k^\eta \sum_{\ell = 1}^{d} w_\ell^{\eta} \right)\le \exp\left(\sum_{k =1 }^{\infty} s_k^\eta \sum_{\ell = 1}^{\infty} w_\ell^{\eta} \right). 
\]
Moreover, it also follows that for any fixed positive integer $i_0$, the sum of the $\lambda_{d,\bsk_i}^\eta$ is bounded below as
\begin{align*}
    \sup_{d \in \N} \, \sum_{i=i_0}^{\infty}\lambda_{d, \bsk_{i}}^{\eta} & \ge w_1^\eta \sum_{k=i_0}^{\infty}s_{\kappa_i}^{\eta} \qquad \parbox{6cm}{considering only $\bsk_i$ of the form $(k, 0, 0, \ldots, 0)$, \\
    and ordering $\bss$ so that $s_{\kappa_1} \ge s_{\kappa_2} \ge \cdots$,} \\
    \sup_{d \in \N} \, \sum_{i=i_0}^{\infty}\lambda_{d, \bsk_{i}}^{\eta} & \ge  s_{1}^{\eta} \sum_{i=i_0}^{\infty} w_{\ell_i}^\eta
    \qquad %\text{for } d \ge i_0 
    \\
    & \qquad \parbox{8cm}{considering only $\bsk_i$ of the form  $(0, \ldots, 0,1,0 , \ldots, 0)$,\\ 
    where the non-zero component is at the $\ell$-th  position with \\
    $i_0\le \ell \le d$, and ordering $\bsw$ so that $w_{\ell_1} \ge w_{\ell_2} \ge \cdots$.}
\end{align*}
Thus, we have necessary and sufficient conditions for strong tractability.
\begin{corollary} \label{DHKM:sptexample_cor}
For the $\lambda_{d,\bsk}$ of the form \eqref{DHKM:prodwts} we have strong polynomial tractability if and only if there exists $\eta>0$ such that
 \begin{equation*}
%\label{DHKM:sptexample}
    \sum_{k=1}^\infty s_{k}^\eta < \infty \text{ and } 
    \sum_{\ell = 1}^{\infty} w_\ell^{\eta} < \infty.
\end{equation*}
\end{corollary}

\begin{remark}
Note that in the setting of this example, the term 
$\sum_{i=1}^{\infty}\lambda_{d, \bsk_{i}}^{\eta}$ will usually depend exponentially on $d$ unless the 
coordinate weights decay to zero fast enough with increasing $\ell$. Hence, we can only hope for tractability under the presence of decaying $w_\ell$. For further details on weighted approximation
problems and tractability, we refer to \cite{NovWoz08a}.
\end{remark}

%%%%%%%%%%%%%%%%%%%%%%%%%%%%%%%%%%%%%%%%%%%%%%%%%%%%%%%%%%%%%%%%%%%%%%%%%%%%%%%%%%%%%%%%%%%%%%%%
%%%%%%%%%%%%%%%%%%%%%%%%%%%%%%%%%%%%%%%%%%%%%%%%%%%%%%%%%%%%%%%%%%%%%%%%%%%%%%%%%%%%%%%%%%%%%%%%
%%%%%%%%%%%%%%%%%%%%%%%%%%%%%%%%%%%%%%%%%%%%%%%%%%%%%%%%%%%%%%%%%%%%%%%%%%%%%%%%%%%%%%%%%%%%%%%%
\section{Tracking the Decay Rate of the Series Coefficients of the Input Function}
\label{DHKM:tracking_sec} 
%%%%%%%%%%%%%%%%%%%%%%%%%%%%%%%%%%%%%%%%%%%%%%%%%%%%%%%%%%%%%%%%%%%%%%%%%%%%%%%%%%%%%%%%%%%%%%%%
%%%%%%%%%%%%%%%%%%%%%%%%%%%%%%%%%%%%%%%%%%%%%%%%%%%%%%%%%%%%%%%%%%%%%%%%%%%%%%%%%%%%%%%%%%%%%%%%
%%%%%%%%%%%%%%%%%%%%%%%%%%%%%%%%%%%%%%%%%%%%%%%%%%%%%%%%%%%%%%%%%%%%%%%%%%%%%%%%%%%%%%%%%%%%%%%%

From error bound \eqref{DHKM:APP_Err_Coef} it follows that the faster the $\hf(\bsk_i)$ decay, the faster $\APP(f,n)$ converges to the solution.  Unfortunately, adaptive Algorithm \ref{DHKM:PilotConeAlg} does not adapt to the decay rate of the $\hf(\bsk_i)$ as $i \to \infty$. It simply bounds $\norm[\calF]{f}$ based on a pilot sample.  The algorithm presented in this section tracks the rate of decay of the $\hf(\bsk_i)$ and terminates sooner if the $\hf(\bsk_i)$ decay more quickly.  Similar algorithms for quasi-Monte Carlo integration are developed in \cite{HicJim16a}, \cite{JimHic16a}, and \cite{HicEtal17a}.

There is an implicit assumption in this section that function data are cheap and we can afford a large sample size.  A large sample size is required to do meaningful tracking of the decay of the series coefficients.  The previous section and the next section are more suited to the case when function data are expensive and the final sample size must be modest.

Let $(n_j)_{j\ge 0}$ be a strictly increasing sequence of non-negative integers.  This sequence may increase geometrically or algebraically. Define the sets of wavenumbers analogously to \eqref{DHKM:KOnedef},
\begin{equation*}
   n_{-1}=0, \qquad \calK_j := \{\bsk_{n_{j-1}+1}, \ldots, \bsk_{n_j}\} \quad \text{for } j \in \N_0.
\end{equation*}
If $n_0 = 0$, then $\calK_0$ is empty.  For any $f \in \calF$, define the norms of subsets of series coefficients:
\begin{equation} \label{DHKM:SigmaDef}
\sigma_j (f):=\norm[\rho]{\biggl(\frac{\hf(\bsk)}{\lambda_{\bsk}} \biggr)_{\bsk \in \calK_j}} \qquad \text{for } j \in \N.
\end{equation}
Thus, $\norm[\calF]{f} = \norm[\rho]{\bigl(\sigma_j(f) \bigr)_{j \in \N_0}}$. 

For this section, we define the cone of input functions by
\begin{equation} \label{DHKM:TrackConeDef}
  \calC : =\left\{f\in\calF \colon \sigma_{j+r} (f)\le ab^r \sigma_j (f)\ \forall j,r\in\NN\right\}.
\end{equation}
Here, $a$ and $b$ are positive reals with $b< 1 < a$. The constant $a$ is an inflation factor, and the constant $b$ defines the general rate of decay of the $\sigma_j(f)$ for $f \in \calC$. Because $ab^r$ may be greater than one, we do not require the series coefficients of the solution, $\SOL(f)$, to decay monotonically. However, we expect their partial sums to decay steadily.  The series coefficients for  wavenumbers $\bsk \in \calK_0$ do not affect the definition of $\calC$ and may behave erratically.
Lemma \ref{DHKM:Key_Lem} implies that 
\begin{equation} \label{DHKM:LambdaDef}
    \norm[\tau]{\bigl(\hf(\bsk) \bigr)_{\bsk \in \calK_j}} \le \sigma_j(f) \Lambda_j, \qquad \text{where } \Lambda_j : = \norm[\rho']{\bigl(\lambda_{\bsk} \bigr)_{\bsk \in \calK_j}}.
\end{equation}
From \eqref{DHKM:SOLNorm} and \eqref{DHKM:SOLNormFinite} it follows that the norm of the solution operator is 
\begin{equation} \label{DHKM:NormLambdaFinite}
    \norm[\calF \to \calG]{\SOL} = \bignorm[\rho']{\bigl(\Lambda_j \bigr)_{j \in \N_0}} < \infty
\end{equation}

If $f$ belongs to the $\calC$ defined in \eqref{DHKM:TrackConeDef} and $n_0 = 0$, then 
\begin{align*}
    \norm[\calF]{f} & = \norm[\rho]{\Biggl(\,\norm[\rho]{\biggl(\frac{\hf(\bsk)}{\lambda_{\bsk}} \biggr)_{\bsk \in \calK_j}}\,\Biggr)_{j \in \N}} = \bignorm[\rho]{\bigl(\sigma_j(f) \bigr)_{j \in \N}} \\
    & \le \bignorm[\rho]{\bigl(\sigma_1(f), ab \sigma_1(f), ab^2 \sigma_1(f), \ldots \bigr)} \\
    & = \left(1 + \frac{a^\rho b^\rho}{1 - b^\rho} \right)^{1/\rho}  \norm[\rho]{\biggl(\frac{\hf(\bsk)}{\lambda_{\bsk}} \biggr)_{\bsk \in \calK_1}}.
\end{align*}
Comparing this inequality to the definition of $\calC$ in the previous section, it can be seen that  $\calC$ defined in \eqref{DHKM:TrackConeDef} is a subset of  $\calC$ defined in \eqref{DHKM:pilot_cone} if we choose 
$A=\left(1 + \frac{a^\rho b^\rho}{1 - b^\rho} \right)^{1/\rho}$ in \eqref{DHKM:pilot_cone}.

From the expression for the error in \eqref{DHKM:APP_Err_Coef} and the definition of the cone in  \eqref{DHKM:TrackConeDef}, we can now derive a data-driven error bound for all $f \in \calC$ and $j \in \bbN$: 
\begin{align}
\nonumber
\MoveEqLeft{\norm[\calG]{\SOL(f)-\APP(f,n_j)}} \\
\nonumber &= \norm[\tau]{\left(\hf(\bsk_i) \right)_{i = n_j+1}^\infty}
= \norm[\tau]{ \left(\norm[\tau]{\bigl(\hf(\bsk) \bigr)_{\bsk \in \calK_l}} \right)_{l=j+1}^\infty}
\\
\nonumber
& \le \norm[\tau]{ \bigl(\sigma_l(f) \Lambda_l \bigr)_{l=j+1}^\infty} \qquad \text{by \eqref{DHKM:LambdaDef}} \\
\nonumber 
&
= \norm[\tau]{ \bigl(\sigma_{j+r}(f) \Lambda_{j+r} \bigr)_{r=1}^\infty}
\\
& \le a \sigma_j(f) \norm[\tau]{ \bigl(b^r\Lambda_{j+r} \bigr)_{r=1}^\infty} =:\ERRNj
 \qquad \text{by \eqref{DHKM:TrackConeDef}.}
 \label{DHKM:algoineq}
\end{align}

This upper bound depends only on the function data and the parameters defining $\calC$.  The error vanishes as $j \to \infty$ because $\sigma_j(f) \le ab^{j-1} \sigma_1(f) \to 0$ and $\Lambda_j \to 0$.  Moreover, the error bound for $\APP(f,n_j)$ depends on $\sigma_j(f)$, whose rate of decay need not be postulated in advance.

These assumptions accommodate both the cases where the approximation converges algebraically and exponentially.  To illustrate the algebraic case, suppose that $\hf(\bsk_i)/\lambda_{\bsk_i} = \calO(i^{-r_\Delta})$ for some positive $r_\Delta > 1/\rho$.  For this algebraic case one would normally define $\calC$ in terms of an exponentially increasing sequence, $(n_j)_{j\ge 0}$, e.g., $n_j = n_0 2^j$, which implies that 
\begin{align*}
    \sigma_j(f) &= \left[ \sum_{i=n_0 2^{j-1} + 1}^{n_0 2^j} \biggabs{\frac{\hf(\bsk_i)}{\lambda_{\bsk_i}}}^\rho \right]^{1/\rho}
    = \calO \left( \left [ \sum_{i=n_0 2^{j-1} + 1}^{n_0 2^j} i^{-\rho r_\Delta} \right]^{1/\rho} \right) \\
    & = \calO \left(  2^{-j(r_\Delta-1/\rho)} \right).
\end{align*}
Reasonable functions would satisfy 
\begin{equation*}
    C_{\lo} 2^{-j(r_\Delta-1/\rho)} \le \sigma_j(f) \le C_{\up} 2^{-j(r_\Delta-1/\rho)} 
\end{equation*}
for some constants $C_{\lo}$ and $C_{\up}$.  Choosing $a \ge C_{\up}/C_{\lo} $  and $b \ge  2^{-(r_\Delta-1/\rho)}$ causes the cone $\calC$ to include such functions.  Note that only the ratio of $C_{\up}$ to $C_{\lo}$ need be assumed to determine $a$, and choosing $b$ larger than necessary does not affect the order of the decay of the error bound. 

To illustrate the exponential case, suppose that $\hf(\bsk_i)/\lambda_{\bsk_i} = \calO(\E^{-r_\Delta i})$.  For this exponential case one would normally define $\calC$ in terms of an arithmetic sequence, $(n_j)_{j\ge 0}$, e.g., $n_j = n_0 + j s$, where $s$ is a positive integer.  This implies that 
\begin{align*}
    \sigma_j(f) &= \left[ \sum_{i=n_0 + j s -s + 1}^{n_0 + j s} \biggabs{\frac{\hf(\bsk_i)}{\lambda_{\bsk_i}}}^\rho \right]^{1/\rho}
    = \calO \left( \left [ \sum_{i=n_0 + j s -s + 1}^{n_0 + j s} \E^{- \rho r_\Delta i} \right]^{1/\rho} \right) \\
    & = \calO \left(  \E^{-j r_{\Delta} s} \right).
\end{align*}
Analogous to the algebraic case, reasonable functions would satisfy 
    $C_{\lo} \E^{-j r_{\Delta} s} \le \sigma_j(f) \le C_{\up} \E^{-j r_{\Delta} s}$
for some constants $C_{\lo}$ and $C_{\up}$.  Choosing $a \ge C_{\up}/C_{\lo} $  and $b \ge \E^{- r_{\Delta} s}$ causes the cone $\calC$ to include such functions.  Again, only the ratio of $C_{\up}$ to $C_{\lo}$ need be assumed to determine $a$, and choosing $b$ larger than necessary does not affect the order of the decay of the error bound.

%%%%%%%%%%%%%%%%%%%%%%%%%%%%%%%%%%%%%%%%%%%%%%%%%%%%%%%%%%%%%%%%%%%%%%%%%%%%%%%%%%%%%%%%%%%%%%%%
\subsection{The Adaptive Algorithm and Its Computational Cost} \label{DHKM:SecAdapAlgTrackDecay}
%%%%%%%%%%%%%%%%%%%%%%%%%%%%%%%%%%%%%%%%%%%%%%%%%%%%%%%%%%%%%%%%%%%%%%%%%%%%%%%%%%%%%%%%%%%%%%%%

The data-driven error bound in \eqref{DHKM:algoineq} forms the basis for an adaptive Algorithm \ref{DHKM:TrackConeAlg}, which solves our problem for input functions in the cone $\calC$ defined in \eqref{DHKM:TrackConeDef}.  The following theorem establishes its viability and computational cost. In deriving upper bounds on the computational cost and lower bounds on the complexity, we may sacrifice tightness for simplicity.

\begin{algorithm}
	\caption{Adaptive ALG for a Cone of Input Functions Tracking the Series Coefficient Decay Rate \label{DHKM:TrackConeAlg}}
	\begin{algorithmic}
	\PARAM the Banach spaces $\calF$ and $\calG$, including the weights $\bslambda$; a strictly increasing sequence of non-negative integers, $(n_j)_{j\ge 0}$; an inflation factor, $a$; the general decay rate, $b$; $\APP$ satisfying \eqref{DHKM:APP_Err_Coef}
		\INPUT a black-box function, $f$; an absolute error tolerance,
		$\varepsilon>0$

\Ensure Error criterion \eqref{DHKM:err_crit} for  the cone defined in \eqref{DHKM:TrackConeDef}

\State Let $j \leftarrow 0$
\Repeat

\State Let $j \leftarrow j + 1$

\State Compute $\ERRNj$ as defined in \eqref{DHKM:algoineq}

\Until $\ERRNj \le \varepsilon$

\RETURN $\ALG(f,\varepsilon) = \APP(f,n_{j})$
\end{algorithmic}
\end{algorithm}

\begin{theorem}\label{DHKM:TractConeCompCost}
Algorithm \ref{DHKM:TrackConeAlg} yields an answer satisfying absolute error criterion \eqref{DHKM:err_crit}, i.e., $\ALG \in \calA(\calC)$ for $\calC$ defined in \eqref{DHKM:TrackConeDef}.  The information cost is $\COST(\ALG,f,\varepsilon)=n_{j^*}$, where $j^*$ is defined implicitly as
\begin{equation} \label{DHKM:EqTractConejstar}
j^* = \min\left \{ j \in \bbN : \ERRNj \le \varepsilon  \right\}.
\end{equation}
Moreover, $\COST(\ALG,\calC,\varepsilon,R) \le n_{j^\dagger}$, where $j^\dagger$ is defined as follows:
\begin{equation} \label{DHKM:TractConejdagger}
j^\dagger = \min \left \{j \in \bbN :   \norm[\tau]{ \bigl(b^{j+r}\Lambda_{j+r} \bigr)_{r=1}^\infty}
\le  \frac{b\varepsilon}{Ra^2} \left( \frac{1 - b^{j\rho}}{1 - b^\rho} \right)^{1/\rho} \right\}.
\end{equation}
\end{theorem}

\begin{proof}
The value of $j^*$ in \eqref{DHKM:EqTractConejstar} follows directly from the error criterion. The success of the algorithm follows from the error bound in \eqref{DHKM:algoineq}.

For the remainder of the proof consider $R$ and $\varepsilon$ to be fixed.  For any $f \in  \calC \cap \calB_R$ and for any $j^\dagger$ defined  as in \eqref{DHKM:TractConejdagger}, it follows that
\begin{align*}
R &\ge \norm[\calF]{f} = \norm[\rho]{\left(\frac{\hf(\bsk)}{\lambda_{\bsk}} \right)_{\bsk \in \bbK}}
 \ge \norm[\rho]{\left(\sigma_j(f)\right)_{j=1}^{j^\dagger}}  
 \qquad \text{by \eqref{DHKM:SigmaDef} } \\
& \ge \norm[\rho]{\left(a^{-1}b^{1-j^\dagger}\sigma_{j^\dagger}(f), \ldots, a^{-1} b^{-1}\sigma_{j^\dagger}(f), \sigma_{j^\dagger}(f) \right) } \quad \text{by  \eqref{DHKM:TrackConeDef}}\\
& \ge \frac {\sigma_{j^\dagger}(f)} a \norm[\rho]{\bigl(b^{1-j^\dagger}, \ldots, b^{-1}, 1 \bigr) } 
= \frac {b\sigma_{j^\dagger}(f)} a \left( \frac{b^{-j^\dagger\rho} -1}{1 - b^\rho} \right)^{1/\rho}\\
& \ge \sigma_{j^\dagger}(f)\,\frac{Ra}{\varepsilon} \norm[\tau]{ \bigl(b^r\Lambda_{j^\dagger+r} \bigr)_{r=1}^\infty}
\qquad \text{by the definition of } j^\dagger \text{ in \eqref{DHKM:TractConejdagger}} \\
& = \frac{R}{\varepsilon} \ERRNjd\, .
\end{align*}
From this last inequality, it follows that $j^\dagger \ge j^*$.
\end{proof} \

Although Algorithm \ref{DHKM:TrackConeAlg} tracks the decay rate of the $\hf(\bsk_i)$, the information cost bound and complexity bound in the theorem above do not reflect different decay rates of the $\hf(\bsk_i)$. That is a subject for future investigation.

%%%%%%%%%%%%%%%%%%%%%%%%%%%%%%%%%%%%%%%%%%%%%%%%%%%%%%%%%%%%%%%%%%%%%%%%%%%%%%%%%%%%%%%%%%%%%%%%
\subsection{Essential Optimality of the Algorithm}
%%%%%%%%%%%%%%%%%%%%%%%%%%%%%%%%%%%%%%%%%%%%%%%%%%%%%%%%%%%%%%%%%%%%%%%%%%%%%%%%%%%%%%%%%%%%%%%%

To establish the essential optimality of Algorithm \ref{DHKM:TrackConeAlg} requires some additional, reasonable assumptions on the sequences $(n_j)_{j \in \N_0}$ and $\bigl(\sigma_j(f) \bigr)_{j \in \N_0}$.  Recall from \eqref{DHKM:NormLambdaFinite} that $\bigl(\Lambda_j \bigr)_{j \in \N_0}$ has a finite $\rho'$ norm.  We require that the $\Lambda_j$ must decay steadily with $j$:
\begin{equation} \label{DHKM:LambdaDecayCond}
    \alpha^{-1} \beta^r \Lambda_j \le \Lambda_{j+r} \le \alpha \gamma^r \Lambda_j  \quad \forall j,r \in \N_0, \qquad \text{for some } \beta, \gamma < 1 \le \alpha.
\end{equation}
We also assume that the ratio of the largest to smallest $\lambda_{\bsk}$ in a group is bounded above:
\begin{equation} \label{DHKM:MinMaxCond}
    \sup_{j \in \N} \frac{\lambda_{\bsk_{n_{j-1}+1}}}{\lambda_{\bsk_{n_{j}}}} \le S_1 < \infty.
\end{equation}
For the illustrative choices of $(n_j)_{j \in \N_0}$ and $\bigl(\bsk_i \bigr)_{i \in \N}$ preceding Section \ref{DHKM:SecAdapAlgTrackDecay} this assumption holds.  Let $\card(\cdot)$ denote the cardinality of a set. We assume that if $\calJ$ is an arbitrary set of wavenumbers with $\card(\calJ) \le n_j$, then there exists some $l \le n_{j+1}$ for which $\calK_l \setminus \calJ$ retains some significant fraction of the original $\calK_l$ elements: 
\begin{equation} \label{DHKM:PropCond}
     \inf_{j \in \N} \ \min_{\calJ \subset \bbK \, : \, \card(\calJ) \le n_j} \ \max_{0 \le l \le j+1} \frac{\card(\calK_l \setminus \calJ)}{\card(\calK_l)} \ge S_2 > 0.
\end{equation}
Again, for the illustrative choices of $(n_j)_{j \in \N_0}$ and $\bigl(\bsk_i \bigr)_{i \in \N}$ preceding Section \ref{DHKM:SecAdapAlgTrackDecay} this assumption holds.

The following theorem establishes a lower bound on the complexity of our problem for input functions in $\calC$. The theorem after that shows that the cost of our algorithm as given in Theorem \ref{DHKM:TractConeCompCost} is essentially no worse than this lower bound.

\begin{theorem} \label{DHKM:TractConeLowBdComp}
A lower bound on the complexity of the linear problem is
\begin{align*}
 %\label{compbdA}
&\COMP(\calA(\calC),\varepsilon,R) > n_{j^\ddagger}, 
\intertext{where}
%\label{compbdB}
j^\ddagger & = \max \left \{ j \in \bbN :  b^{j+1} \Lambda_{j+1}    > 
 \frac{2a\alpha \varepsilon}{R(a-1)(1 - b^\rho)^{1/\rho}}  \left[1 + \left(\frac 1 {S_2} -1 \right) S_1^\rho \right]^{1/\rho}
\right \}.
\end{align*}
\end{theorem}

\begin{proof}

As in the proof of Theorem \ref{DHKM:PilotCostThm} we consider fixed and arbitrary R and $\varepsilon$.
We proceed by carefully constructing the test input functions, $f_1$ and $f_{\pm} = f_1 \pm f_2$, lying in $\calC \cap \calB_{R}$, which yield the same approximate solution but different true solutions.  This leads to a lower bound on $\COMP(\calA(\calC),\varepsilon,R)$. The proof is provided for $\rho' < \infty$.  The proof for $\rho' = \infty$ is similar.

The first test function $f_1 \in \calC$ is defined in terms of its series coefficients---inspired by Lemma \ref{DHKM:Key_Lem}---as
\begin{align}
\nonumber
f_1 &= f_{10} + f_{11} +  \cdots, \qquad
\hf_{1j}(\bsk) := \begin{cases}
\displaystyle
\frac{c_1 b^{j} \lambda_{\bsk}^{\rho'/\rho+1}}{\Lambda_j^{\rho'/\rho}},  & \bsk \in \calK_j,
\\
0, & \bsk \notin \calK_j,
\end{cases}
\\
\nonumber
c_1 &:=  \frac{R(a+1)(1 - b^\rho)^{1/\rho}}{2a}.
\end{align}
It can be verified that the test function lies both in $\calB_{R}$ and in $\calC$:
\begin{align}
\nonumber
\sigma_j(f_1) & = \norm[\rho]{\biggl(\frac{\hf_{1j}(\bsk)}{\lambda_{\bsk}} \biggr)_{\bsk \in \calK_l}} 
= c_1 b^j, \qquad j \in \N_0,\\
\nonumber
\norm[\calF]{f_1} &= \norm[\rho]{\bigl( \sigma_j(f) \bigr)_{j \in \N_0} } 
=  \frac{c_1}{(1 - b^\rho)^{1/\rho}} = \frac{R(a+1)}{2a} \le R,
\\
\nonumber
\sigma_{j+r}(f_1) &= 
b^{r} \sigma_j(f_1) \le a b^r \sigma_j(f_1), \qquad j,r \in \N_0.
\end{align}

Now let $\ALG'$ be an arbitrary algorithm in $\calA(\calC)$, and suppose that $\ALG'(f_1,\varepsilon)$ samples $f_1(\bsk)$ for $\bsk \in \calJ$.  Let $\tcalK_j = \calK_j \setminus \calJ$ for all non-negative integers $j$. Construct the function $f_2$, having zero coefficients for $\bsk \in \calJ$, but otherwise looking like $f_1$:
\begin{align}
\nonumber
f_2 &= f_{20} + f_{21} +  \cdots, \qquad \hf_{2j}(\bsk) := \begin{cases}
\displaystyle
\frac{c_2 b^{j} \lambda_{\bsk}^{\rho'/\rho+1}}{\tLambda_j^{\rho'/\rho}},  
& \bsk \in \tcalK_j,
\\
0, & \text{otherwise},
\end{cases}
\\
\nonumber
c_2 &:= \frac{R(a-1)(1 - b^\rho)^{1/\rho}}{2a}, \qquad
\tLambda_j := \norm[\rho']{\bigl(\lambda_\bsk \bigr)_{\bsk \in \tcalK_j}} \le \Lambda_j, \\
\nonumber
\sigma_j(f_2) & = \norm[\rho]{\biggl(\frac{\hf_{2j}(\bsk)}{\lambda_{\bsk}} \biggr)_{\bsk \in \tcalK_j}} 
= \begin{cases} c_2 b^j, & \tcalK_j \ne \emptyset, \\
0, & \tcalK_j = \emptyset, 
\end{cases}
\qquad j \in \N_0, \\
\nonumber 
\norm[\calF]{f_2} &= \norm[\rho]{\bigl( \sigma_j(f_{2}) \bigr)_{j \in \N_0} } 
\le \frac{c_2}{(1 - b^\rho)^{1/\rho}} = \frac{R(a - 1)}{2a}\le R, \\
\nonumber 
\norm[\calG]{\SOL(f_{2j})} &= \sigma_j(f_{2j}) \tLambda_j = 
c_2 b^{j} \tLambda_j, \qquad j \in \N_0, \\
\norm[\calG]{\SOL(f_2)} & = \norm[\tau]{\bigl(c_2 b^{j} \tLambda_j \bigr)_{j \in \N_0}}
= \frac{R(a-1)(1 - b^\rho)^{1/\rho}}{2a} \norm[\tau]{\bigl(b^{j} \tLambda_j \bigr)_{j \in \N_0}}.
\label{DHKM:SOLftwo}
\end{align}

Furthermore, define $f_{\pm} = f_1 \pm f_2$.
It can be verified that $f_{\pm}$ also lie both in $\calB_{R}$ and in $\calC$:
\begin{align}
\nonumber
\norm[\calF]{f_\pm} &\le \norm[\calF]{f_1} + \norm[\calF]{f_2} \le \frac{c_1 + c_2}{(1 - b^\rho)^{1/\rho}} = R,
\\
\nonumber
\sigma_j(f_\pm) & \ge \sigma_j(f_1) - \sigma_j(f_2) \ge
\left(c_1 - c_2 \right) b^{j} = \frac{R(1 - b^\rho)^{1/\rho}b^{j}}{a},  \qquad j \in \N_0,
\\
\nonumber
\sigma_{j+r}(f_\pm) & \le \sigma_{j+r}(f_1) + \sigma_{j+r}(f_2) \le 
(c_1+c_2) b^{j+r} 
\\
\nonumber
& =  R(1 - b^\rho)^{1/\rho}b^{j+r}
\le a b^r \sigma_j(f_\pm),  \qquad j \in \N_0.
\end{align}

Since $\hf_2(\bsk) = 0$ for $\bsk \in \calJ$, it follows that $\ALG'(f_\pm,\varepsilon) = \ALG'(f_1,\varepsilon)$.  But, even though the two test functions $f_\pm$ lead to the same approximate solution, they have different true solutions.  In particular,
\begin{align}
\nonumber
\varepsilon &\ge \max \bigl\{\norm[\calG]{\SOL(f_+) - \ALG'(f_+,\varepsilon)}, \norm[\calG]{\SOL(f_-) - \ALG'(f_-,\varepsilon)} \bigr\} \\
\nonumber
&\ge \frac 12 \bigl[\norm[\calG]{\SOL(f_+) - \ALG(f_1,\varepsilon)} + \norm[\calG]{\SOL(f_-) - \ALG'(f_1,\varepsilon)}  \bigr] \\
\nonumber
&\qquad \qquad \text{since } \ALG'(f_\pm,\varepsilon) = \ALG'(f_1,\varepsilon) \\
\nonumber
&\ge \frac 12 \norm[\calG]{\SOL(f_+) - \SOL(f_-)} \quad \text{by the triangle inequality}\\
\nonumber
&\ge \frac 12 \norm[\calG]{\SOL(f_+ - f_-)} \quad \text{since $\SOL$ is linear}\\
&= \norm[\calG]{\SOL(f_2)} 
= \frac{R(a-1)(1 - b^\rho)^{1/\rho}}{2a} \norm[\tau]{\bigl(b^{j} \tLambda_j \bigr)_{j \in \N_0}}
\qquad 
\text{by \eqref{DHKM:SOLftwo}.}
\label{DHKM:eps_LBA}
\end{align}

Suppose that $\card(\calJ) = \COST(\ALG',f_{\pm},\varepsilon) \le n_{j^{\star}}$.  Then by condition \eqref{DHKM:PropCond}, there exists an $l^\star \le j^\star+1$ where $\card(\tcalK_{l^\star}) \ge S_2 \card(\calK_{l^\star})$.  
This implies a lower bound on $\tLambda_{l^\star}$.  Let $m = n_{l^\star} - n_{l^\star-1} = \card(\calK_{l^\star})$.  Then, $m_{\inc} = \lceil S_2 m \rceil \ge S_2 m$ is a lower bound on  $\card(\tcalK_{l^\star})$, and $m_{\out} = m - m_{\inc} \le (1 - S_2) m$ is an upper bound on $\card(\calK_{l^\star} \setminus \tcalK_{l^\star})$.  Moreover, 

\begin{align*}
    \Lambda_{l^\star}^\rho & = \sum_{i \in \calK_{l^\star}} \lambda_{\bsk_i}^{\rho} =  \tLambda_{l^\star}^\rho + \sum_{i \in \calK_{l^\star} \setminus \tcalK_{l^\star}} \lambda_{\bsk_i}^{\rho}
    \\
    &\le  \tLambda_{l^\star}^\rho + m_{\out} \lambda_{\bsk_{n_{l^\star -1}+1}}^\rho \qquad \text{by the ordering of the } \lambda_{\bsk_i}\\
    &\le  \tLambda_{l^\star}^\rho + m_{\out} S_1^\rho \lambda_{\bsk_{n_{l^\star}}}^\rho \qquad \text{by \eqref{DHKM:MinMaxCond}}\\
    & \le \tLambda_{l^\star}^\rho + \frac{m_{\out}}{m_{\inc}} S_1^\rho \tLambda_{l^\star}^\rho \qquad \text{by the definition of } \tLambda_{l^\star}\\
    & \le \left[1 + \left(\frac 1 {S_2} -1 \right) S_1^\rho \right] \tLambda_{l^\star}^\rho \qquad \text{by the bounds on } m_{\inc} \text{ and } m_{\out} 
    \\
    & \le \left[1 + \left(\frac 1 {S_2} -1 \right) S_1^\rho \right] b^{-\rho l^\star} \norm[\tau]{\bigl(b^{j} \tLambda_j \bigr)_{j \in \N_0}}^\rho\, .
\end{align*}

Returning to \eqref{DHKM:eps_LBA}, the above inequality  implies that
\begin{equation*}
    \varepsilon 
\ge  \frac{R(a-1)(1 - b^\rho)^{1/\rho}}{2a} \left[1 + \left(\frac 1 {S_2} -1 \right) S_1^\rho \right]^{-1/\rho} b^{l^\star} \Lambda_{l^\star}.
\end{equation*}
Since $l^\star \le j^{\star}+1$ it follows that $ b^{l^\star} \ge b^{j^\star+1}$ and from condition \eqref{DHKM:LambdaDecayCond} it follows that $\Lambda_{l^\star} \ge \Lambda_{j^\star+1}/\alpha$.  Thus,
\begin{equation*}
    \varepsilon 
\ge  \frac{R(a-1)(1 - b^\rho)^{1/\rho}}{2a\alpha} \left[1 + \left(\frac 1 {S_2} -1 \right) S_1^\rho \right]^{-1/\rho} b^{j^\star+1} \Lambda_{j^\star+1}.
\end{equation*}
If any algorithm satisfies the error tolerance $\varepsilon$ for all input functions in $\calC \cap \calB_R$ and has information cost no greater than $n_{j^\star}$, then $j^\star$ must satisfy the above inequality.  By contrast, if the above inequality is violated for any $j^\star$, then the information cost of the successful algorithm must be greater than $n_{j^\star}$.  This completes the proof.
\end{proof}

\begin{theorem}
\label{DHKM:TrackConeAlgOptThm}
Adaptive Algorithm \ref{DHKM:TrackConeAlg} is essentially optimal for the cone of input functions defined in \eqref{DHKM:TrackConeDef}.
\end{theorem}
\begin{proof}
Let $j^\dagger(\varepsilon)$ be defined as in \eqref{DHKM:TractConejdagger}, with the $\varepsilon$ dependence made explicit.  Choose $\varepsilon_{\max}$ and $R_{\min}$ in \eqref{DHKM:EssentialOpt} such that $j^\dagger(\varepsilon) \ge 2$.  This definition implies that

\begin{align*}
    b^{j^\dagger(\varepsilon)} \Lambda_{j^\dagger(\varepsilon) } 
   & = \frac{[1 - (\gamma b)^\tau]^{1/\tau}}{\alpha}
   \norm[\tau]{ \bigl(b^{j^\dagger(\varepsilon)-1+r} \alpha \gamma^{r-1} \Lambda_{j^\dagger(\varepsilon)} \bigr)_{r=1}^\infty}
    \\
    & \ge \frac{[1 - (\gamma b)^\tau]^{1/\tau}}{\alpha}  \norm[\tau]{ \bigl(b^{j^\dagger(\varepsilon)-1+r}\Lambda_{j^\dagger(\varepsilon)-1+r} \bigr)_{r=1}^\infty} \qquad \text{by \eqref{DHKM:LambdaDecayCond}} 
    \\
    &
    > \frac{b [1 - (\gamma b)^\tau]^{1/\tau} \varepsilon}{Ra^2 \alpha} \left( \frac{1 - b^{\rho(j^\dagger(\varepsilon)-1)}}{1 - b^\rho} \right)^{1/\rho} \qquad \text{by \eqref{DHKM:TractConejdagger}}
    \\
    &
    \ge \frac{b [1 - (\gamma b)^\tau]^{1/\tau} \varepsilon}{Ra^2 \alpha} \qquad \text{since } j^\dagger(\varepsilon) \ge 2
    \\
    & =   \frac{\alpha^2}{(b\beta)^2} \times \frac{2a \alpha \omega \varepsilon}{R(a-1)(1 - b^\rho)^{1/\rho}}  \left[1 + \left(\frac 1 {S_2} -1 \right) S_1^\rho \right]^{1/\rho}
    \\
    \intertext{where }
   \omega  & = \frac{(a-1)b^3 \beta^2 (1 - b^\rho)^{1/\rho} [1 - (\gamma b)^\tau]^{1/\tau}}{2a^3 \alpha^4 } \left[1 + \left(\frac 1 {S_2} -1 \right) S_1^\rho \right]^{-1/\rho}.
\end{align*}

Making the $\varepsilon$ dependence explicit in the definition of $j^\ddagger(\varepsilon)$ in Theorem \ref{DHKM:TractConeLowBdComp} it follows from the above inequality that
\begin{equation*}
    b^{j^\dagger(\varepsilon)} \Lambda_{j^\dagger(\varepsilon) } 
     > \frac{\alpha^2}{(b\beta)^2} b^{j^\ddagger(\omega\varepsilon)+2} \Lambda_{j^\ddagger(\omega \varepsilon) + 2 } 
    \ge \alpha b^{j^\ddagger(\omega\varepsilon)} \Lambda_{j^\ddagger(\omega \varepsilon)} \qquad \text{by \eqref{DHKM:LambdaDecayCond}}.
\end{equation*}

If $j^{\dagger}(\varepsilon) \ge j^{\ddagger}(\omega \varepsilon)$, then \eqref{DHKM:LambdaDecayCond} implies that 
\[
 b^{j^\dagger(\varepsilon)} \Lambda_{j^\dagger(\varepsilon) } \le \alpha (\gamma b)^{j^{\dagger}(\varepsilon) - j^{\ddagger}(\omega \varepsilon)} b^{j^{\ddagger}(\omega \varepsilon) } \Lambda_{j^\ddagger(\omega \varepsilon) } \le \alpha b^{j^{\ddagger}(\omega \varepsilon) } \Lambda_{j^\ddagger(\omega \varepsilon) }.
\]
But, this contradicts the above inequality.  Thus, $j^\dagger(\varepsilon) < j^\ddagger(\omega \varepsilon)$, and so
\[
\COST(\ALG,\calC,\varepsilon,R) \le n_{j^\dagger(\varepsilon)} < n_{j^\ddagger(\omega \varepsilon)} < \COMP(\calA(\calC),\omega \varepsilon,R).
\]
Thus,  Algorithm \ref{DHKM:TrackConeAlg} is essentially optimal.
\end{proof}

%%%%%%%%%%%%%%%%%%%%%%%%%%%%%%%%%%%%%%%%%%%%%%%%%%%%%%%%%%%%%%%%%%%%%%%%%%%
\subsection{Tractability}\label{DHKM:SecDecayTract}
%%%%%%%%%%%%%%%%%%%%%%%%%%%%%%%%%%%%%%%%%%%%%%%%%%%%%%%%%%%%%%%%%%%%%%%%%%%

We again would like to study tractability. As it turns out, by using the relation 
between the cones defined in \eqref{DHKM:pilot_cone} and \eqref{DHKM:TrackConeDef}, respectively, we 
easily obtain sufficient conditions for tractability. 

\begin{theorem} \label{DHKM:thmtract3}
The respective conditions presented in Theorem \ref{DHKM:thmtract1} for the case where $\rho'=\infty$ and 
in Theorem \ref{DHKM:thmtract2} for the case where $\rho'<\infty$ are sufficient for strong polynomial, polynomial, and weak tractability of the approximation problem defined on cones as in \eqref{DHKM:TrackConeDef}.
\end{theorem}
\begin{proof}
As pointed out above, $\calC$ defined in \eqref{DHKM:TrackConeDef} is a subset of  $\calC$ defined in \eqref{DHKM:pilot_cone}, by choosing $A=\left(1 + \frac{a^\rho b^\rho}{1 - b^\rho} \right)^{1/\rho}$ in \eqref{DHKM:pilot_cone}. This means that the approximation problem on  $\calC$ defined in \eqref{DHKM:TrackConeDef} is essentially (i.e., up to constants depending on $A,a,b$ and $\rho$) no harder than the same problem on $\calC$ defined in \eqref{DHKM:pilot_cone}. This, however, implies that all sufficient conditions in Theorem \ref{DHKM:thmtract1} are also sufficient in the case considered in Theorem \ref{DHKM:thmtract3}.
\end{proof}

Theorem \ref{DHKM:thmtract3} yields sufficient conditions for the tractability notions considered 
here. A general result for necessary conditions seems to be more difficult to obtain and is left open for future research.

%%%%%%%%%%%%%%%%%%%%%%%%%%%%%%%%%%%%%%%%%%%%%%%%%%%%%%%%%%%%%%%%%%%%%%%%%%%%%%%%%%%%%%%%%%%%%%%%%%%%%%
%%%%%%%%%%%%%%%%%%%%%%%%%%%%%%%%%%%%%%%%%%%%%%%%%%%%%%%%%%%%%%%%%%%%%%%%%%%%%%%%%%%%%%%%%%%%%%%%%%%%%%
\section{Inferring Coordinate and Smoothness Importance} \label{DHKM:smoothimportance_sec}
%%%%%%%%%%%%%%%%%%%%%%%%%%%%%%%%%%%%%%%%%%%%%%%%%%%%%%%%%%%%%%%%%%%%%%%%%%%%%%%%%%%%%%%%%%%%%%%%%%%%%%
%%%%%%%%%%%%%%%%%%%%%%%%%%%%%%%%%%%%%%%%%%%%%%%%%%%%%%%%%%%%%%%%%%%%%%%%%%%%%%%%%%%%%%%%%%%%%%%%%%%%%%

In Sections \ref{DHKM:pilot_sec} and \ref{DHKM:tracking_sec}, the weights $\bslambda = (\lambda_{\bsk})_{\bsk \in \bbK}$, which appear in the definition of the cone of inputs, $\calC$, are taken as given and fixed.  One may assume the form suggested in \eqref{DHKM:prodwts}, which defines $\bslambda$ in terms of coordinate weights and smoothness weights.  However, practically speaking it may be difficult to know a priori the values of these weights.
This section explores a situation where the initial data collected for the input function data can be used to learn $\bslambda$, inferring which input variables in $f$ may be more important and the smoothness of the function.  

The motivation for this section is situations where the relative importance of the $d$ input variables of the function is not known from physical considerations. We also envision situations where the cost of function data is large, e.g., the result of an expensive computer simulation.  Thus, we are not concerned with the cost of the algorithm beyond the information cost, which we hope to limit to $\calO(d)$.

\subsection{Product, Order and Smoothness Dependent (POSD) Weights}
The $u_{\bsk}$ and the $\lambda_{\bsk}$ in this section are defined as 
\begin{equation}
u_\bsk = \prod_{\ell = 1}^d \tu_{k_\ell}, \quad 
\lambda_{\bsk} = \Gamma_{\|\bsk\|_0} \prod_{\substack{\ell=1\\ k_\ell>0}}^d w_\ell s_{k_\ell}, \quad \Gamma_0 = s_1 = 1, \quad \bsk \in \mathbb{N}_0^d,
\label{DHKM:posdeq}
\end{equation}
where $\bsw = (w_\ell)_{l=1}^d$ is the vector of coordinate weights, $\bss = (s_k)_{k=1}^\infty$ is the vector of smoothness weights,  $\boldsymbol{\Gamma} = (\Gamma_m)_{m=1}^d$ is the vector of \emph{order weights}, and $\norm[0]{\bsk}$ denotes the number of nonzero elements of $\bsk$. The intuition behind these weights is as follows: 
\begin{itemize}
    \item Coordinate weights quantify the importance for the $d$ input variables in $f$.
    \item Smoothness weights quantify the importance of the $\tu_k$.  E.g., if the $\tu_k$ are polynomials of degree $k$ as in Section \ref{DHKM:secexamp}, then the faster the $s_k$ decay, the smoother $f$ is.  
    \item Order weights quantify the importance of effects with different orders; $\bsk$ having one nonzero element corresponds to a first-order or main effect, $\bsk$ having two nonzero elements corresponds to a second-order (interaction) effect.   (e.g., first-order, second-order).
\end{itemize} \ \vspace{-4ex}

This parametrization is motivated by several guiding principles from the experimental design literature \cite{WuHam2009}, which are briefly described below.  In statistical parlance, the terms $\hf(\bsk) u_{\bsk}$ are effects.
\begin{itemize}
\item \emph{Effect sparsity} assumes that only a small number of inputs in $f$ are important. In \eqref{DHKM:posdeq}, this sparsity means that only a small number of product weights $\bsw$ are large.  This principle arises in the sufficient condition for strong tractability in Corollary \ref{DHKM:sptexample_cor}.
\item \emph{Effect heredity} assumes that lower-order effects are more important than higher-order effects. E.g., $\lambda_{(1,0, 0,\ldots, 0)}$ should be larger than $\lambda_{(1, 1, 0, \ldots, 0)}$. In \eqref{DHKM:posdeq}, this heredity can be enforced by assuming that the order weights $\Gamma_{m}$ decrease with $m$. 
\item \emph{Effect hierarchy} assumes that an effect is active \emph{only} when all its component effects are active. For example, $\lambda_{(1, 1, 0, \ldots, 0)} > 0$ only when $\lambda_{(1, 0, 0, \ldots, 0)}$ and $\lambda_{(0, 1, 0, \ldots, 0)}$ are both nonzero. This hierarchy is implicitly enforced by the product structure of the weights in \eqref{DHKM:posdeq}.
\item \emph{Effect smoothness} assumes that lower-degree effects are more important than higher-degree effects. For example, when the $(\tu_k)_{k \in \mathbb{N}_0}$ are polynomials, this means that linear effects are more important than quadratic effects, which are in turn more significant than cubic effects, and so on. Effect smoothness can be imposed by assuming $\bss$ to be a decreasing sequence. 
\end{itemize} \ \vspace{-4ex}

The $\lambda_{\bsk}$ defined in \eqref{DHKM:posdeq} are called product, order and smoothness dependent (POSD) weights.  From a quasi-Monte Carlo (QMC) perspective, the POSD weights in \eqref{DHKM:posdeq} generalize upon the product-and-order dependent (POD) weights in \cite{KuoEtal12a}, which were introduced for analyzing QMC methods in partial differential equations with random coefficients. The latter POD weights can be recovered by ignoring the smoothness weights. 

Our POSD weights differ from the smoothness-driven product-and-order dependent (SPOD) weights in \cite{Dea2014}, which were recently used to analyze higher-order QMC methods for stochastic partial differential equations. These SPOD weights take the form:
\begin{equation*}
\gamma_{\fraku} = \sum_{\bsk \in \{1, \ldots, \alpha\}^{|\fraku|}} \|\bsk\|_1 ! \prod_{\ell \in \fraku} \left( 2^{\delta(k_\ell,\alpha)} w_\ell^{k_\ell} \right), \quad \|\bsk\|_1 = \sum_{l=1}^d k_\ell, \quad \fraku \subseteq \{1, \ldots, d\},
%\label{eq:spod}
\end{equation*}
where $\delta (k_\ell,\alpha)$ is 1 if $k_\ell=\alpha$ and 0 otherwise. 
Intuitively, the SPOD weights quantify the importance of each \textit{subspace} (indexed by $\fraku$), under a common smoothness structure among subspaces (for further details on SPOD weights, we refer the reader to \cite{Dea2014}). In contrast, the proposed POSD weights in \eqref{DHKM:posdeq} instead quantify the importance of each \textit{Fourier series coefficient} $\hf(\bsk)$ (indexed by $\bsk$), under a common smoothness structure among coefficients.

\subsection{Inferring POSD Weights from an Initial Sample}

Let $\calC_{\bslambda}$ denote the cone of inputs defined in \eqref{DHKM:pilot_cone} by POSD weights $\bslambda = \bigl( \lambda_{\bsk} \bigr)_{\bsk \in \N_0^d}$.  As mentioned above, our goal here is to infer $\bslambda$ from input function data.  We start with an initial set of wavenumbers:
\begin{equation} \label{DHKM:barKdef}
    \bcalK= \{ (0, \ldots, 0, k, 0, \ldots, 0): k = 0, \ldots, k_{\max}\}.
\end{equation}
The approximation to $f$ based on sampling the series coefficients for these wavenumbers is
\begin{equation*}
    \fapp = \sum_{\bsk \in \bar{\calK}} \hf(\bsk) u_{\bsk}.
\end{equation*}
We choose the $\calC_{\bslambda}$ that best fits $f$ by selecting $\bslambda$ to make the norm of  $f_{\text{app}}$ small:
\begin{multline}
\bar{\bslambda} = \bslambda(\bar{\bsw}, \bar{\bss}, \bsGamma), \\
\text{where } (\bar{\bsw}, \bar{\bss} ) = \min \left\{ \argmin_{(\bsw,\bss) \in \calW \times \calS} \left\|\left( \frac{\hf(\bsk)}{\lambda_{\bsk}(\bsw,\bss,\bsGamma)} \right)_{\bsk \in \bar{\calK}}\right\|_{\rho} \right \}.
\label{DHKM:eq:inf}
\end{multline}
Here, $\calW$ is a candidate set for coordinate weights, e.g., $\calW = [0,w^*]^d$, and $\calS$ is a candidate set for the smoothness weights, e.g., $\calS = \{(1/k^r)_{k=1}^{\infty} \colon r > 0\}$. The inner minimization finds the $(\bsw,\bss)$ that minimizes the approximate norm of the input function.  This minimizer may be non-unique, so the outer minimization chooses the smallest such $(\bsw,\bss)$. Making the coordinate and smoothness weights as small as possible helps enforce the principles of effect sparsity.  The optimum, $(\bar{\bsw}, \bar{\bss})$, then defines the \emph{data-inferred} POSD $\bslambda$, denoted $\bar{\bslambda}$.

The candidate sets $\calW$ and $\calS$ should be constructed such that the coordinate and smoothness weights have a priori upper bounds. Otherwise the inner minimization would choose huge values for $\bsw$ and $\bss$ to maximize the $\lambda_{\bsk}(\bsw,\bss,\bsGamma)$ and minimize the norm of $\fapp$.  The cardinality of the initial set of wavenumbers is $d k_{\max} + 1$.  There is a trade-off between keeping $k_{\max}$ small enough to reducing cost and making $k_{\max}$ large enough to ensuring robustness.

For simplicity, we assume that order weights, $\bsGamma$, are fixed a priori.  If desired, they too could be inferred as the next step.  However, since we want to limit the size of the initial sample to $\calO(d)$ we must sample judiciously the higher order interactions.  

The optimization in \eqref{DHKM:eq:inf} is nontrivial to solve numerically. In practice, we iteratively optimize over $\bsw$ and then $\bss$ until convergence is reached.  At each step of the iteration $\|(\hat{f}(\bsk)/\lambda_{\bsk})_{\bsk \in \bcalK}\|_{\rho}$ decreases.

Algorithm \ref{DHKM:InfPilotConeAlg} combines the construction of data-inferred POSD weights, $\bar{\bslambda}$, with Algorithm \ref{DHKM:PilotConeAlg} of Section \ref{DHKM:pilot_sec}. This algorithm succeeds for input functions in the cone 
\begin{equation} \label{DHKM:pilot_cone2}
    \bar{\calC} := \left \{ f \in \calF : f \in \calC_{\bar{\bslambda}} \text{ for } \bar{\bslambda} \text{ defined in \eqref{DHKM:eq:inf}} \right\}.
\end{equation}
The reason that $\bar{\calC}$ is a cone is that the data-inferred $\bar{\bslambda}$ for the input function $f$ is exactly the same as for the input function $cf$, where $c$ is any constant.

\begin{algorithm}
	\caption{Adaptive $\ALG$ Based on Data-Inferred POSD Weights \label{DHKM:InfPilotConeAlg}} 
	\begin{algorithmic}
	\PARAM the bases $\{u_{\bsk}\}_{\bsk \in \N_0^d}$ and $\{v_{\bsk}\}_{\bsk \in \N_0^d}$; candidate sets $\calW$ and $\calS$; maximum smoothness degree, $k_{\max}$; order weights, $\boldsymbol{\Gamma}$; an inflation factor, $A > 1$; $\APP$ satisfying \eqref{DHKM:Refined_APP_err} 
		\INPUT a black-box function, $f$; an absolute error tolerance,
		$\varepsilon>0$

\Ensure Error criterion \eqref{DHKM:err_crit} for the cone defined in \eqref{DHKM:pilot_cone2}
\State  Define the initial set of  wavenumbers $\bcalK$ defined in \eqref{DHKM:barKdef}
\State Evaluate initial sample $\bigl\{\hf(\bsk)\bigr\}_{\bsk \in \bcalK}$
\State Compute data-driven POSD weights, $\bar{\bslambda}$ according to \eqref{DHKM:eq:inf}
\State Using these weights, $\bar{\bslambda}$, perform Algorithm \ref{DHKM:PilotConeAlg} to obtain $\ALG(f,\varepsilon)$
\RETURN $\ALG(f,\varepsilon)$
\end{algorithmic}
\end{algorithm}

\subsection{Numerical Examples} \label{DHKM:numexamp_sec}

We now investigate the numerical performance of this adaptive algorithm using data-inferred POSD weights. For simplicity, only the case of $\rho = \infty$ and $\rho' = \tau = 1$ is considered in the following examples. Here, the basis functions $(u_\bsk)_{\bsk \in \mathbb{N}_0^d}$ are Chebyshev polynomials in Section \ref{DHKM:secexamp}, and the solution operator is $\SOL (f) = f$ (i.e., function approximation). We note that $\|f - \ALG(f,\varepsilon)\|_\infty \le \|f - \ALG(f,\varepsilon)\|_\calG$, so our error criterion \eqref{DHKM:err_crit} implies that $\|f - \ALG(f,\varepsilon)\|_\infty \le \varepsilon$.

The simulation set-up is as follows. The Fourier coefficients for input function $f$, $\{\hat{f}(\bsk)\}_{\bsk \in \mathbb{N}_0^d}$, are randomly sampled as:
\begin{equation*}
\hat{f}(\bsk) = Z_{\bsk} \,  {\Gamma}_{\|\bsk\|_0}^{\rm tr} \prod_{\substack{\ell=1\\ k_\ell>0}}^d {w_\ell^{\rm tr}} {s}_{k_\ell}^{\rm tr}, \quad Z_{\bsk} \overset{i.i.d.}{\sim} \text{Unif}[-1, 1], \quad \bsk \in \mathbb{N}_0^d.
%\label{eq:foursim}
\end{equation*}
Here, $(w_\ell^{\rm tr})_{\ell=1}^d = (1/L^2(\ell))_{l=1}^d$, $({\Gamma_k^{\rm tr}})_{k=1}^\infty \equiv 1$ and $(s_j^{\rm tr})_{j=1}^{k_{\rm max}} = (1/j^4)_{j=1}^4$ are the true coordinate, order, and smoothness weights, and $Z_{\bsk}$ randomly sets the magnitude and sign of each coefficient. Moreover, $\bigl(L(\ell)\bigr)_{\ell=1}^d$ is a random permutation of $1, \ldots, d$ to ensure that the order of input variables does not necessarily reflect their order of importance. We also set $\boldsymbol{\Gamma} = \boldsymbol{\Gamma}^{\rm tr}$ in Algorithm \ref{DHKM:InfPilotConeAlg} and use an inflation factor of $A = 1.1$.

Figures \ref{fig:four} (a) and (b) display the total required sample size from Algorithm \ref{DHKM:InfPilotConeAlg}, as a function of the error to tolerance ratio, $\|f - \ALG(f,\varepsilon)\|_\infty/\varepsilon$, in $d=4$ and $d=7$ dimensions, respectively. Each data point corresponds to a different error tolerance $\varepsilon$. A ratio $\|f - \ALG(f,\varepsilon)\|_\infty/\varepsilon$ close to, but not exceeding, one is desired, since this shows that our adaptive algorithm is successful. For $d=4$, $\|f - \ALG(f,\varepsilon)\|_\infty/\varepsilon$ fluctuates around 0.4 for all choices of $\varepsilon$; for $d=7$, this ratio begins at $\approx 0.1$ for $\varepsilon = 0.1$, then decreases to $\approx 0.014$ for $\varepsilon = 0.001$. This shows that our adaptive approximation algorithm works reasonably well.  It appears slightly more effective in lower dimensions  than in higher dimensions. A likely reason is that the underlying POSD structure can be more easily learned from a small pilot sample in lower dimensions than in higher dimensions.

\begin{figure}
\centering
\begin{subfigure}{\textwidth}
\centering
\includegraphics[width=0.7\textwidth]{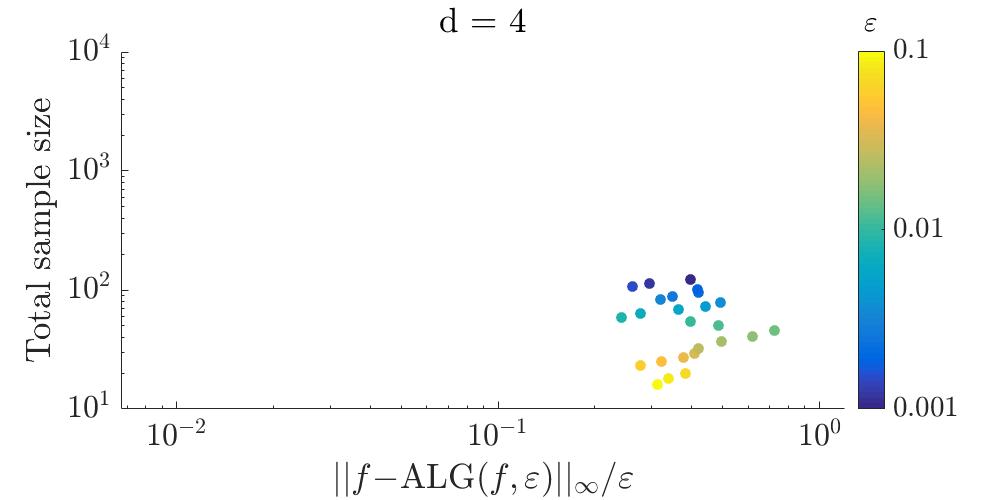}
\label{fig:four1}
\caption{$f$ is a $d=4$-dim. function with random Fourier coefficients.}
\end{subfigure}
\begin{subfigure}{\textwidth}
\centering
\includegraphics[width=0.7\textwidth]{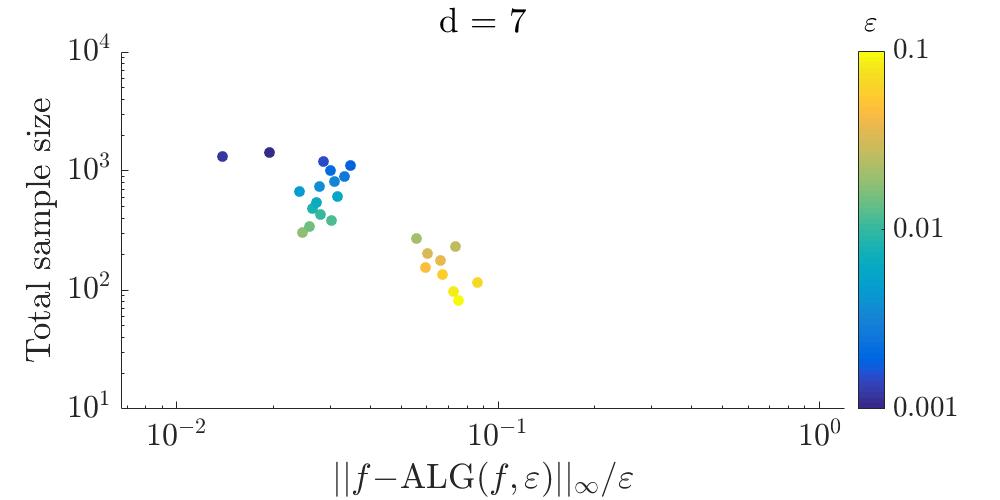}
\label{fig:four2}
\caption{$f$ is a $d=7$-dim. function with random Fourier coefficients.}
\end{subfigure}
\caption{Total required sample size as a function of error ratio $\|f - \ALG(f,\varepsilon)\|_{\infty}/\varepsilon$, with points colored by the absolute error tolerance level $\varepsilon$.}
\label{fig:four}
\end{figure}

\section*{Acknowledgement}
F.~J.~Hickernell, P.~Kritzer, and S.~Mak gratefully acknowledge the support of the Statistical and Applied Mathematical Sciences Institute year-long ``Program on Quasi-Monte Carlo and High-Dimensional Sampling Methods for Applied Mathematics'' through NSF-DMS-1638521.  F.~J.~Hickernell also acknowledges the support of NSF-DMS-152268.
F.~J.~Hickernell and P.~Kritzer thank the RICAM Special Semester Program 2018 for support. P.~Kritzer gratefully acknowledges support by the Austrian Science Fund (FWF) Project  F5506-N26, which is part of the Special Research Program ``Quasi-Monte Carlo Methods: Theory and Applications''.

\bibliographystyle{plain}
\bibliography{ExtraBib.bib,FJH23.bib,FJHown23.bib}

\bigskip

\bigskip

\begin{small}
\noindent\textbf{Authors' addresses:}\\

 \noindent Yuhan Ding\\ 
 Department of Mathematics\\ 
 Misericordia University\\ 
 301 Lake Street, Dallas, PA 18704 USA\\
  
 \medskip 
  
 \noindent Fred J. Hickernell\\
 Department of Applied Mathematics\\ 
 Illinois Institute of Technology\\ 
 RE 220, 10 W.\ 32${}\text{nd}$ Street, Chicago, IL 60616 USA\\

 \medskip 
 
 \noindent Peter Kritzer\\
 Johann Radon Institute for Computational and Applied Mathematics (RICAM)\\
 Austrian Academy of Sciences\\
 Altenbergerstr. 69, 4040 Linz, Austria\\
 
 \medskip
 
 \noindent Simon Mak\\
 H. Milton Stewart School of Industrial and Systems Engineering\\ 
 Georgia Institute of Technology\\ 
 755 Ferst Drive, Atlanta, GA 30332 USA\\

 \end{small}

\end{document}